\documentclass[10pt,a4paper]{article}
\usepackage[utf8]{inputenc}
\usepackage{amsmath}
\usepackage{amsfonts}
\usepackage{amssymb}
\usepackage{amsthm}
\usepackage{enumitem}
\usepackage{color}
\usepackage{hyperref}
\usepackage{MnSymbol}
\usepackage{graphicx}

\usepackage[a4paper, textwidth=13cm]{geometry}

\newcommand{\N}{\mathbb{N}}
\newcommand{\R}{\mathbb{R}}
\newcommand{\Z}{\mathbb{Z}}
\renewcommand{\oe}{\Omega_{\epsilon}}
\renewcommand{\ge}{\Gamma_{\epsilon}}
\newcommand{\oet}{\Omega_{\epsilon}(t)}
\newcommand{\get}{\Gamma_{\epsilon}(t)}
\newcommand{\seps}{S_{\epsilon}}
\newcommand{\jeps}{J_{\epsilon}}
\newcommand{\ie}{i.\,e.\,}

\newcommand{\tueps}{\tilde{u}_{\epsilon}}

\newcommand{\tqeps}{\tilde{q}_{\epsilon}}

\newcommand{\ueps}{u_{\epsilon}}
\newcommand{\deps}{D_{\epsilon}}
\newcommand{\qeps}{q_{\epsilon}}
\newcommand{\veps}{v_{\epsilon}}
\newcommand{\weps}{w_{\epsilon}}

\newcommand{\oeh}{\Omega_{\epsilon}^h}
\newcommand{\geh}{\Gamma_{\epsilon}^h}

\newcommand{\he}{\mathcal{H}_{\epsilon}}
\newcommand{\hoeh}{\mathcal{H}_0^1(\Omega_{\epsilon}^h)}
\newcommand{\tphi}{\widetilde{\phi}}
\newcommand{\oeth}{\Omega_{\epsilon}^{2h}}
\newcommand{\fxe}{\frac{x}{\epsilon}}

\newcommand{\te}{\mathcal{T}_{\epsilon}}
\newcommand{\tbe}{\mathcal{T}_{\epsilon}^b}
\newcommand{\ue}{\mathcal{U}_{\epsilon}}
\newcommand{\hoy}{\mathcal{H}_0^1(Y^{\ast})}

\newcommand{\per}{\mathrm{per}}
\newcommand{\peps}{\phi_{\epsilon}}
\newcommand{\tuo}{\tilde{u}_0}

\newtheorem{definition}{Definition}
\newtheorem{remark}{Remark}
\newtheorem{theorem}{Theorem}
\newtheorem{proposition}{Proposition}
\newtheorem{lemma}{Lemma}
\newtheorem{corollary}{Corollary}
\newtheorem{example}{Example}

\title{Homogenization of a reaction-diffusion-advection problem in an evolving micro-domain and including nonlinear boundary conditions}

\author{M. Gahn\thanks{Interdisciplinary Center for Scientific Computing, University of Heidelberg,   Heidelberg, Germany. \textit{Mail: markus.gahn@iwr.uni-heidelberg.de}} \and  M. Neuss-Radu\thanks{Department of Mathematics, Friedrich-Alexander-Universit\"at Erlangen-N\"urnberg,  Erlangen, Germany. \textit{Mail: maria.neuss-radu@math.fau.de }} \and I. S. Pop\thanks{Hasselt University, Faculty of Sciences, Agoralaan Gebouw D, Diepenbeek 3590, Belgium. \textit{Mail: sorin.pop@uhasselt.be}}}
\date{}

\begin{document}

\maketitle

\begin{abstract}
We consider a reaction-diffusion-advection problem in a perforated medium, with nonlinear reactions in the bulk and at the microscopic boundary, and slow diffusion scaling. The microstructure changes in time; the microstructural evolution is known \textit{a priori}. The aim of the paper is the rigorous derivation of a homogenized model. We use appropriately scaled function spaces, which allow us to show compactness results, especially regarding the time-derivative and we prove strong two-scale compactness results of Kolmogorov-Simon-type, which allow to pass to the limit in the nonlinear terms. The derived macroscopic model depends on the micro- and the macro-variable, and the evolution of the underlying microstructure is approximated by time- and space-dependent reference elements.
\end{abstract}



\section{Introduction}

In this paper, we consider a reaction-diffusion-advection problem in a perforated medium with evolving microstructure (the micro-domain). Such type of problems are encountered as mathematical models for mineral precipitation or biofilm growth in porous media \cite{Noorden2009, Bringedal2014, Schulzetal2016, SchulzKnabner}, manufacturing of steel \cite{EdenMuntean17}, or the swelling of mitochondria within biological cells \cite{SchefflerMitochondria}. 
Typically, in such mathematical models two spatial scales can be encountered: a microscopic scale representing e.g. the scale of pores in a porous medium and at which the processes can be described in detail, and a macroscopic one, representing e.g. the scale of the full domain. 

In situations like presented above, the microscopic processes can alter significantly the microstructure, which changes in time. The evolution of the micro-domain then depends on unknown quantities like the concentration of the transported species, or the fluid flow and the pressure. On the other hand, these unknowns do depend on the microstructural evolution, so one speaks about problems involving free boundaries at the microscopic scale. From a numerical point of view, such mathematical models are extremely complex, which makes simulations a challenging task. 

In this context, a natural approach is to derive upscaled, macroscopic mathematical models approximating the microscopic ones, and describing the averaged behaviour of the quantities. Compared to the models involving a fixed microstructure, the case in which the microscopic sub-domains occupied by the fluid and by the solid phases are separated by a moving interface requires additionally a way to express the microstructural evolution in time. One possibility is to use phase fields as smooth approximations of the characteristic functions of the two above-mentioned sub-domains (see e.g. \cite{RRP, BvWP} for a reactive transport model in a porous medium involving precipitation and dissolution). This has the advantage that the moving boundaries are approximated by diffuse transition regions and that the microscopic model is defined in a fixed domain. On the other hand, this introduces curvature effects in the evolution of this diffuse region and, implicitly, of the evolving interface. Furthermore, two different limits have to be considered: one in which the diffuse interface is reduced to a sharp one, and another in which the macroscopic model is obtained  from the microscopic one. Choosing the proper balance in this limit process is not trivial, in particular in a mathematically rigorous derivation. 

Alternatively, one can consider sharp interfaces separating the phase sub-domains. In the one-dimensional case, or for simple geometries like a (thin) strip or a radially symmetric tube, the distance to the lateral boundary can be used to locate the moving interface. For more complex situations like perforated domains, the moving boundaries can be defined as the 0-level set of a function solving a level set equation involving the other model unknowns. In this sense we refer to the micro-model proposed in \cite{Noorden2009} for crystal precipitation and dissolution in a porous medium. The evolution of microscopic precipitate layer is described through a level-set, and a macroscopic model is derived by a formal asymptotic expansion. In \cite{Bringedal2014}, this procedure is extended to account for non-isothermal effects. Similarly, in \cite{Schulzetal2016} the level set is used to describe the microscopic growth of a biofilm in a porous medium, and the corresponding macroscopic model is derived by formal homogenization techniques. 

For giving a mathematically rigorous derivation of the macroscopic model, and more precise of the  convergence of the limiting process, the main difficulties are in proving the existence of a solution for the microscopic, free boundary model, and in obtaining a priori estimates that are sufficient for a two-scale limit. In a first step towards the rigorous homogenization of problems involving free boundaries at the micro scale, here we consider a simplified situation and assume that the microstructural evolution is known \textit{a priori}. In other words,  we assume that the movement of the interface can be expressed through a time and space dependent mapping from a reference domain (see also \cite{Peter07, Peter09, EdenMuntean17}). 

In what follows, $\epsilon > 0$ is a small parameter describing the ratio between the typical micro-scale length and the one of the $n$-dimensional hyper-rectangle $\Omega$. It also represents the periodicity parameter of a fixed, periodically perforated domain $\oe$. 
At each time $t\in [0, T]$, the micro-domain (occupied by the fluid) is $\oet$, with boundary $\partial \oet$. Both are obtained as the image of a given mapping $\seps: [0, T] \times \overline{\Omega} \rightarrow \R^n$, namely $\oet = \seps(t, \oe)$ and $\partial \oe(t) = \seps(t, \partial \oe)$. Here, we assume that $\seps$ is a diffeomorphism and sufficiently regular with bounds uniformly with respect to $\epsilon$, see Assumptions \ref{AssumptionsAprioriSeps}  - \ref{AssumptionConvergenceTransformation} and Example \ref{Beispiel}.

Inside the moving domain $\oet$ the microscopic model is a reaction-diffusion-advection equation. The diffusion is assumed of order $\epsilon^2$, and the advection of order $\epsilon$. 
The aim of this paper is a mathematically rigorous derivation of a macroscopic model, of which solution approximates the one of the microscopic model. This is done by using rigorous multi-scale techniques such as two-scale convergence \cite{Nguetseng,Allaire_TwoScaleKonvergenz,Neuss_TwoScaleBoundary,AllaireDamlamianHornung_TwoScaleBoundary} and the unfolding method \cite{Cioranescu_Unfolding1,CioranescuDamlamianDonatoGrisoZakiUnfolding}, also known as periodic modulation or dilation operator \cite{ArbogastDouglasHornung,BourgeatLuckhausMikelic,VogtHomogenization}. In  doing so, the main challenge is to pass to the limit in the nonlinear terms. Hence, a crucial part of the paper is the derivation of strong two-scale compactness results based only on estimates for the solution of the micro-model.

Since the microstructural evolution is assumed known, to derive the macroscopic model we transform the microscopic problem, defined in the moving domain $\oet$, to a problem defined on the fixed, reference domain $\oe$. 
This leads to a change in the coefficients of the equations, which now depend on the transformation $\seps$ between $\oe$ and $\oet$. Especially, the time-derivative is applied to the product $\jeps \ueps$ involving the determinant of the Jacobian of $\seps$, $\jeps := \det(\nabla \seps)$. We mention that $\epsilon$-independent estimates are only obtained for the time-derivative $\partial_t(\jeps \ueps)$ and not for $\partial_t \ueps$, and therefore the product $\jeps \ueps$ is used to control the dependence of $\ueps$ on the time variable in the homogenization process. Additionally, the regularity of the product $\jeps \ueps$ with respect to time is insufficient to control the nonlinear boundary terms, and therefore refined arguments involving the time-derivative are required for the derivation of strong compactness results. A further challenge in the proof of the strong convergence is due to the small diffusion coefficient (of order $\epsilon^2$), which leads to oscillations of the solution with respect to  the spatial variable and thus to a gradient norm of order $\epsilon^{-1}$. Hence, the two-scale limit depends on a macro- and a micro-variable, and standard compactness results used  in the case of  diffusion coefficients of order $1$, like the extension of the solution to the whole domain, see \cite{Acerbi1992}, and Aubin-Lions-type compactness arguments, see \cite{MeirmanovZimin}, fail. To overcome these problems, we use the unfolding operator and prove strong convergence of the unfolded sequence in the $L^p$- sense by using a Banach-valued compactness argument of Kolmogorov-Simon type. These are based on \textit{a priori} estimates for the differences of the shifts with respect to the spatial variable and the control of the time variable via the time-derivative. For the latter, we prove a commuting property for the generalized time-derivative and the unfolding operator using  a duality argument, see also \cite{GahnNeussRaduKnabner2018a} where similar techniques were used for reaction-diffusion problems through thin heterogeneous layers. 

The derived effective model depends on both, the micro- and the macro-variable. The macro-variable lies in the fixed domain $\Omega$. In every macro-point of $x \in \Omega$, one has to solve a local cell problem with respect to the micro-variable, and on an evolving reference element. For such cell problems, the macro-variable $x$ enters rather as a parameter. Hence, the evolution of the micro-structure is passed on to the cell problems in the effective model.

Strictly referring to rigorous homogenization results for problems with an evolving microstructure, we mention that even if the microscopic evolution is known \textit{a priori}, there are only a few results available and these are for \textit{linear} problems, at least in the boundary terms. In this sense we mention \cite{Peter07}, dealing with a linear reaction-diffusion-advection problem in a two-phase medium, and with diffusion of order $\epsilon^2$ and $1$ in the different phases. Similarly, a model related to thermoelasticity in a two-phase domain was considered in \cite{EdenMuntean17}. In \cite{Peter09}, a reaction-diffusion problem with an evolving volume fraction (in our notation $\jeps$) is treated, where the volume fraction is described by an ordinary differential equation which is not coupled to the solution of the reaction-diffusion equation. 
A model for chemical degradation in multi-component evolving porous medium is considered in  \cite{peter2009multiscale} using a formal asymptotic expansion. The homogenization of a linear parabolic equation in a domain with a partly evolving and rapidly oscillating boundary is considered in \cite{muthukumar2018homogenization}.
In \cite{NoordenMuntean2011, NoordenMuntean2013} a macroscopic model is derived for a linear double porosity model in a locally periodic medium not evolving in time. The case of a locally periodic perforated domain that is also time dependent is considered in \cite{fotouhi2020homogenization}. There the homogenization of a linear Robin boundary value problem is analyzed by means of two-scale asymptotic expansion and corrector estimates. The homogenization of a linear parabolic equation in a perforated domain with rapidly pulsating (in time) periodic perforations, with a homogeneous Neumann condition on the boundary of the perforations was considered in  \cite{cioranescu2003homogenization}. This is extended in \cite{cioranescu2006homogenization} to the heat equation defined in a perforated medium with random rapidly pulsating perforations. 
A problem on a fixed domain, with diffusion of order $\epsilon^2$ and nonlinear bulk kinetics - however, without the coefficient $\jeps$ in the time-derivative - is treated in \cite{MielkeReicheltThomas2014}. The authors determine the strong convergence by estimating the difference between the unfolded equation and the solution to the macroscopic equation, which has to be known \textit{a priori}. However, an additional regularity and compatibility condition for the initial data is required. In \cite{Gahn,GahnNeussRaduKnabnerEffectiveModelSubstrateChanneling} a reaction-diffusion system for a two-component connected-disconnected medium was considered for fast diffusion (of order 1) and nonlinear interface conditions as in our problem, where in \cite{GahnNeussRaduKnabnerEffectiveModelSubstrateChanneling} an additional surface equation has been taken into account.

The paper is organized as follows: in Section \ref{SectionMicroscopicModel}, we introduce the microscopic model in an evolving domain and transform it to a fixed domain. In Section \ref{SectionAprioriEstimates}, the \textit{a priori} estimates used for the derivation of the macro-model are proved. In Section \ref{StrongTwoScaleResult}, we establish general strong two-scale compactness results and a convergence result for the time-derivative. The macroscopic model is derived in Section \ref{DerivationMacroscopicModel}. The paper ends with the conclusions in Section \ref{SectionConclusion}. 

\subsection{Original contributions}

Here we address the homogenization of a reactive transport model in an evolving perforated domain. The evolution of the microstructure is assumed known \textit{a priori}. Since nonlinear bulk and surface reactions are considered, the derivation of strong multi-scale compactness results is essential. The main results are:
\begin{enumerate}
[label = - ]
\item The derivation of \textit{a priori} estimates for the micro-problem with coefficients depending on the transformation between the time-dependent and the fixed domain, and especially of the estimates for  the differences between the shifted solution and the solution itself, see Lemma \ref{AprioriEstimates} and \ref{ApriorEstimatesShifts}; 
\item The introduction of the space $\he$ with a weighted Sobolev-norm adapted to the slow diffusion scaling (of order $\epsilon^2$), which is the basis for weak and strong two-scale compactness results, see Section \ref{StrongTwoScaleResult}; 
\item The commuting property of the generalized time-derivative and the unfolding operator, based on a duality argument, see Proposition \ref{DerivativeUnfoldingOperator};
\item The general strong two-scale compactness result of Kolmogorov-Simon-type, see Theorem \ref{StrongTwoScaleCompactness}; 
\item The two-scale compactness result for the generalized time-derivative, see Proposition \ref{TimeDerivativeTSConvergence}; 
\item The derivation of a homogenized model for a reaction-diffusion-advection equation in an evolving domain, with nonlinear reaction-kinetics and low regularity for the time-derivative, see Section \ref{DerivationMacroscopicModel}.
\end{enumerate}

The resulting homogenized model \eqref{HomogenizedModelEvolving} is a reaction-diffusion-advection equation involving a macro-variable $x\in \Omega$ and a micro-variable $y \in Y(t,x)$. More precisely, for every time $t \in [0,T]$ and for every macroscopic point $x \in \Omega$, $Y(t,x)$ denotes a reference element. Observe that $Y(\cdot ,x)$ evolves in time and that the macro-variable $x$ only acts as a parameter. Hence, a parabolic problem has to be solved in every point $x \in \Omega$, in the evolving cell $Y(\cdot ,x)$. The evolution of the micro-domain $\oet$ in the microscopic model enters in the homogenized model via the reference cells $Y(t,x)$. The macro-domain $\Omega$ remains fixed, since we only consider (small) deformations of order $\epsilon$ of the micro-cells and of the outer boundary.

\section{The microscopic model}
\label{SectionMicroscopicModel}

With $n \in \N$, $n > 1$ and $a,b \in \N^n$ such that $a_i < b_i$ for all $i = 1, \dots n$, we consider the hyper-rectangle $\Omega = (a,b)\subset \R^n$ as the macroscopic domain. Further $\epsilon >0$ is a small parameter such that $\epsilon^{-1} \in \N$. 
$\oe \subset \Omega$ is a fixed, periodically perforated microscopic domain, constructed as follows. With $Y:=(0,1)^n$, we let $Y^{\ast} \subset Y$ be a connected subdomain with Lipschitz-boundary and opposite faces matching each other, \ie for $i=1,\ldots, n$ it holds that
\begin{align*}
\partial Y^{\ast} \cap \{x_i = 0\} + e_i = \partial Y^{\ast} \cap \{x_i = 1\}.
\end{align*}
We define $\Gamma:= \partial Y^{\ast} \setminus \partial Y$ and assume that $\Gamma $ is a Lipschitz boundary. Further, let $K_{\epsilon}:= \left\{k \in \Z^n\, : \, \epsilon(Y + k) \subset \Omega \right\}$. Clearly, $\Omega = \mathrm{int} \left(\bigcup_{k \in K_{\epsilon}}  \epsilon \big(\overline{Y} + k \big) \right)
$. Now we define $\oe$ by
\begin{align*}
\oe:= \mathrm{int} \left(\bigcup_{k \in K_{\epsilon}}  \epsilon \big(\overline{Y^{\ast}} + k \big) \right)
\end{align*}
and the oscillating boundary $\ge$ as 
\begin{align*}
\ge := \partial \oe \setminus \partial \Omega.
\end{align*}
We assume that $\oe $ is connected and has Lipschitz-boundary. We emphasize that the complement $\Omega \setminus \oe$ may be connected (if $n\geq 3$) or disconnected. 

For   $t \in [0,T]$, the evolving domain $\oet$ and the evolving surface $\get$ are described by a mapping $\seps: [0, T] \times \overline{\Omega} \to \R^n$,
\begin{align*}
\seps(t,\cdot): \overline{\oe} \rightarrow \overline{\oet},
\end{align*}
with $\get = \seps(t,\ge)$, see Figure \ref{FigurTransformation}. The Jacobi determinant of $\seps$ is denoted by $\jeps$, \ie we have $\jeps := \det \big(\nabla \seps\big)$. 

\begin{figure}
\includegraphics[scale=0.63]{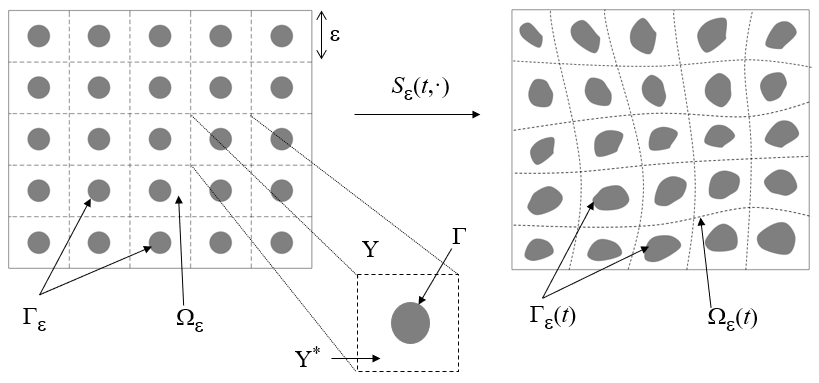}
\caption{The fixed domain $\oe$  and the time-dependent domain $\oet$ obtained through the mapping $\seps(t, \cdot)$.}
\label{FigurTransformation}
\end{figure}

Now we define the non-cylindrical domains $Q_{\epsilon}^T$ and $G_{\epsilon}^T$ by
\begin{align*}
Q_{\epsilon}^T := \bigcup_{t\in (0,T)} \{t\} \times \oet, \quad\quad G_{\epsilon}^T := \bigcup_{t\in (0,T)} \{t\} \times \get, 
\end{align*}
and consider \textbf{Problem P}, namely to find $\tueps$ satisfying:
\begin{align}
\begin{aligned}
\label{MicProblemEvolvingDomain}
\partial_t \tueps -  \nabla \cdot \big(\epsilon^2 D \nabla \tueps  - \epsilon \tqeps \tueps\big) &= f(\tueps) &\mbox{ in }& Q_{\epsilon}^T
\\
-   \epsilon^2 D \nabla \tueps  \cdot \nu &= -\epsilon g(\tueps ) &\mbox{ on }& G_{\epsilon}^T,
\\
-   \epsilon^2 D \nabla \tueps    \cdot \nu &= 0 &\mbox{ on }& \bigcup_{t\in (0,T)} \{t\} \times \partial \oet \setminus \get,
\\
\tueps(0)&= \tueps^0 &\mbox{ in }& \oe(0).
\end{aligned}
\end{align}
Here, $\nu$ denotes the outer unit normal with respect to $\oet$, $\epsilon^2 D$ is the diffusion-tensor, and  $\epsilon \tqeps : \overline{\oe} \to \R^n$ is a material velocity with the property   that at the moving surface $\partial \oet$ it is equal to the evolution of the surface, $\epsilon \tqeps(t,\cdot) \cdot \nu = \partial_t \seps\big(t,\seps(t,\cdot)^{-1}\big) \cdot \nu $. We emphasize that this condition is needed for the derivation of the weak formulation of the problem, but not for the analysis. Therefore it is not stated in Assumption \ref{AssumptionQeps}, referring to $\tqeps$. The functions $f$ and $\epsilon g$ describe the reaction kinetics in the bulk domain $\oet$, respectively at the surface $\get$.
For more details about the derivation of such models we refer to e.g. \cite{AlphonseEtal2018}.

We use common notations in the functional analysis. For a Banach space $X$ with norm $\|\cdot\|_X$ we also use this notation instead of $\|\cdot\|_{X^n}$ for the norm in the product space $X^n = X \times \ldots \times X$. Also, we use $C > 0 $ as a generic constant independent of $\epsilon$. With this, 
and assuming that all integrals are well defined,  we obtain by integration by parts and the Reynolds transport theorem the following weak form of Problem P (as given in \eqref{MicProblemEvolvingDomain}): \\[0.5em]
\textbf{Problem P$_W$}. 
Find $\tueps \in L^2((0,T),H^1(\oe(t)))$  such that for all $\phi \in C^1\big(\overline{Q_{\epsilon}^T}\big)$ with $\phi(T,\cdot) = 0$ it holds that
\begin{align*}
-\int_0^T \int_{\oet}& \tueps \partial_t \phi dx dt + \int_0^T \int_{\oet} \big[ \epsilon^2 D \nabla \tueps - \epsilon \tqeps \tueps \big] \cdot \nabla \phi dx dt 
\\
=&\int_0^T \int_{\oet} f(\tueps) \phi dx dt + \epsilon \int_0^T \int_{\get} g(\tueps) \phi d\sigma dt 
+ \int_{\oe(0)} \tueps^0 \phi(0) dx 
\end{align*}
%

Below we state the 
\noindent\textbf{assumptions on $\seps$:}
\begin{enumerate}
[label = (A\arabic*)]
\item\label{AssumptionsAprioriSeps} $\seps \in C^1([0,T] \times \overline{\oe})^n$ with 
\begin{align*}
\frac{1}{\epsilon}\|\partial_t \seps \|_{C^0([0,T] \times \overline{\oe})} + 
\|\seps\|_{C^0([0,T]\times \overline{\oe})} + \Vert \nabla \seps \Vert_{C^0([0,T]\times \overline{\oe})} &\le C.
\end{align*}
\item There exist constants $c_0, C_0> 0$ independently of $\epsilon$, such that 
\begin{align*}
c_0 \le \jeps \le C_0.
\end{align*}
\item\label{AssumptionBoundsJeps} We have $\jeps \in C^1([0,T]\times \overline{\oe})$ with 
\begin{align*}
\|\partial_t \jeps \|_{L^2((0,T),\he')} + 
\epsilon \|\nabla \jeps \|_{L^{\infty}((0,T)\times \oe)} \le C. 
\end{align*}
For the definition of the space $\he$ we refer to Section \ref{SectionAprioriEstimates}.
\item\label{AssumptionEstimateShiftGradientJeps} For any
$\ell \in \Z^n$ with $\vert \epsilon l \vert < h$ and $0 < h \ll 1$ fixed, it holds that 
\begin{align*}
\|\jeps (\cdot_t, \cdot_x + \ell \epsilon ) - \jeps \|_{C^0([0,T] \times \overline{\oeh})  } &\le C |\ell \epsilon|, \mbox{ and }
\\
\|\nabla \jeps (\cdot_t, \cdot_x + \ell \epsilon ) - \nabla \jeps \|_{L^{\infty}((0,T)\times \oeh) } &\le C |\ell|,
\end{align*}
For the definition of the domain $\oeh$ we refer to Section \ref{SectionAprioriEstimates}. 
\item\label{AssumptionEstimateShiftVelocitySurface} For any $\ell \in \Z^n$ it holds that
\begin{align*}
\Vert \veps (\cdot_t, \cdot_x + \ell\epsilon) - \veps \Vert_{C^0([0,T] \times \overline{\oeh})} \le C \vert \ell \vert \epsilon^2,
\end{align*}
with $\veps(t,x):= \nabla \seps(t,x)^{-1} \partial_t \seps(t,x)$ and $0 < h \ll 1$ fixed.
\item\label{AssumptionConvergenceTransformation} There exists $S_0 \in C^0\big( \overline{\Omega} , C^1\big( [0,T]\times\overline{Y^{\ast}}\big)\big)^n$  such that $S_0(t,x,\cdot_y)$ is $Y$-periodic and  $S_0(t,x,\cdot):Y^{\ast} \rightarrow Y(t,x):= \mathrm{R(S_0(t,x,\cdot))}$ (the range of $S_0(t,x,\cdot_y)$) is a $C^1$-diffeomorphism and (for the definition of the two-scale convergence see Section \ref{StrongTwoScaleResult}) 
\begin{align*}
\seps(t,x) &\rightarrow x &\mbox{ strongly in the two-scale sense,}
\\
\nabla \seps(t,x) &\rightarrow \nabla_y S_0(t,x,y) &\mbox{ strongly in the two-scale sense,}
\\
\nabla \seps^{-1}(t,x) &\rightarrow \nabla_y S_0^{-1}(t,x,y) &\mbox{ strongly in the two-scale sense,}
\\
\epsilon^{-1} \partial_t \seps(t,x,y) &\rightarrow \partial_t S_0(t,x,y) &\mbox{ strongly in the two-scale sense.}
\end{align*}
\end{enumerate}
 Especially, it holds that $\jeps \rightarrow J_0:= \det \nabla_y S_0$ strongly in the two-scale sense. We emphasize that due to the estimates in Assumption \ref{AssumptionsAprioriSeps} the strong two-scale convergences in \ref{AssumptionConvergenceTransformation}  hold with respect to every $L^p$-norm for $p \in [1,\infty)$ arbitrary large. For the derivation of the macroscopic model we need the strong convergence of $\te \jeps$, where $\te $ denotes the unfolding operator, see Section \ref{StrongTwoScaleResult}.

In the following we give an example of a transformation $\seps$, for which the assumptions \ref{AssumptionsAprioriSeps} - \ref{AssumptionConvergenceTransformation} are fulfilled.

\begin{example}\label{Beispiel}
Let us consider that $Y \setminus \overline{Y^{\ast}} $ is strictly included in $Y$, \ie \,$\Gamma$ is a compact subset of $Y$. Furthermore, let $\omega(t,x,y)$ with $\omega: [0,T] \times \overline{\Omega} \times \overline{Y} \rightarrow \R$ be a smooth function, which is $Y$-periodic with respect to variable $y$. We assume that for every $(t,x) \in [0,T] \times \overline{\Omega}$ the set
\begin{align}\label{BeispielEvolution}
\Gamma(t,x) := \left\{ y + \omega(t,x,y) \nu_0(y) \, : \, y \in \Gamma\right\}  \subset Y
\end{align}
is a closed $C^2$-manifold  with $\Gamma(0,x)=\Gamma$. Here $\nu_0$ denotes the outer unit normal vector with respect to $Y^{\ast}$.  Since $\Gamma $ is $C^2$, the outer unit normal $\nu_0$ can be extended to a tubular neighborhood of $\Gamma$, such that it has compact support in this neighborhood and has $C^1$-regularity. The properties of $\Gamma(t,x)$ imply the existence of cubes $W_o$ and $W_i$, such that $W_i \subset W_o \subset Y$ and 
\begin{align*}
\Lambda:= \bigcup_{(t,x)\in [0,T]\times \overline{\Omega}} \Gamma(t,x) \subset W_o \setminus \overline{W_i}.
\end{align*}
Let us define a cut-off function $\chi_0 \in C_0^{\infty}\big(W_o \setminus \overline{W_i}\big)$ with $0 \le \chi_0 \le 1$ and $\chi_0 = 1$  in $\Lambda$. Now, we define the transformation $\seps$ by
\begin{align*}
\seps (t,x):= x + \epsilon \omega\left(t,\left[\fxe\right],\fxe \right) \chi_0 \left(\fxe\right) \nu_0 \left(\fxe\right).
\end{align*}
This means, that the moving interface $\get$ can be described locally on every microscopic cell. With $S_0$ defined by
\begin{align*}
S_0(t,x,y):= y + \omega(t,x,y)\chi_0(y) \nu_0(y),
\end{align*}
the sequence $\seps$ fulfills the properties \ref{AssumptionsAprioriSeps} - \ref{AssumptionConvergenceTransformation}.
\end{example}
 
\begin{remark}\mbox{}
\begin{enumerate}
[label = (\roman*)]
\item In  applications like the ones mentioned in the introduction, the transformation $\seps$ is not known a priori, but is itself a model unknown. For example, when considering dissolution and precipitation in a porous medium, the evolution of $\get$ is determined by the two processes named above, which depend on the concentration of the solute at the pore walls. Conversely, $\oet$ is the domain in which the solute transport model component is defined, so the solute concentration depends on the evolution of the free boundary $\get$. This leads to a microscopic, moving interface problem. From the evolution of $\get$ the transformation $\seps$ can be  constructed for example via the Hanzawa-transformation \cite{Hanzawa81}, see also the monograph \cite{PruessSimonettBook16} for an overview on this topic. Regarding Example \ref{Beispiel}, the crucial point is to find the function $\omega$ and to guarantee the condition $\eqref{BeispielEvolution} $ together with the regularity of $\Gamma(t,x)$. Often this is only possible locally with respect to time, which rises additional difficulties in the homogenization procedure.
\item  The strong regularity assumptions on $\seps$ and the uniform estimates with respect to $\epsilon$ are necessary for the analysis below. A relaxation of the assumptions would make the analysis significantly more difficult, beginning with the \textit{a priori} estimates and the existence of a solution. However, from the application point of view the transformation describes small microscopic deformation, as it occurs for example in mineral precipitation. Especially, Assumptions \ref{AssumptionsAprioriSeps} and \ref{AssumptionConvergenceTransformation} imply that the microscopic deformations taken into account are small at both the micro-scale (the boundary of the perforations) and the macro-scale (the outer boundary). In other words, the displacement of every point with respect to its initial configuration is of order $\epsilon$. Assumptions \ref{AssumptionEstimateShiftGradientJeps}  and \ref{AssumptionEstimateShiftVelocitySurface} say that the shapes of neighboring microscopic cells are similar. Further, from \ref{AssumptionsAprioriSeps} it follows that no clogging occurs.
%
%
\end{enumerate}

\end{remark}

Finally, we state the \textbf{assumptions on the data:}
\begin{enumerate}
[label = (AD\arabic*)]
\item\label{AssumptionQeps} It holds that $\tqeps \in L^{\infty}(Q_{\epsilon}^T)^n$  and 
\begin{align*}
\|\tqeps\|_{L^{\infty}(Q_{\epsilon}^T)} \le C.
\end{align*}
Further, for $\qeps(t,x):= \nabla \seps(t,x)^{-1} \tqeps(t,\seps(t,x))$ for almost every $(t,x) \in (0,T) \times  \oe$
 we assume that for any $\ell \in \Z^n$ it holds that
\begin{align*}
\|\qeps  (\cdot_t, \cdot_x + \ell\epsilon ) - \qeps \|_{L^{\infty}((0,T)\times \oeh)} \le C|\ell\epsilon| .
\end{align*}
There exists a $Y$-periodic function $q_0  \in L^{\infty}((0,T)\times \Omega \times Y^{\ast})$  such that
\begin{align*}
\tqeps \rightarrow \nabla_y S_0^{-1}q_0  \quad \mbox{ in the two-scale sense,}
\end{align*}
see Section \ref{StrongTwoScaleResult} for the definition of two-scale convergence.
\item It holds that $f, \, g \in L^2((0,T)\times \R)$, and for almost every $t \in (0,T)$ the functions $z \mapsto f(t,z)$ and $z \mapsto g(t,z)$ are globally Lipschitz-continuous uniformly with respect to $t$.
\item The diffusion tensor $D \in \R^{n\times n}$ is symmetric and positive.
\item\label{AssumptionInitialConditions}  For $\ueps^0(x)= \tueps^0(\seps(0,x))$ it holds that  $\ueps^0 \in L^2(\oe)$ are bounded uniformly with respect to $\epsilon$ and for any $\ell \in \Z^n$ such that $|\ell\epsilon|<h$ with $0 < h \ll 1$ fixed, it holds that
\begin{align*}
\|\ueps^0(\cdot + \ell\epsilon) - \ueps^0\|_{L^2(\oeh)} \overset{\ell\epsilon \to 0}{\longrightarrow} 0.
\end{align*}
Further, there exists $u^0 \in L^2(\Omega \times Y^{\ast})$, such that $\ueps^0  \rightarrow u^0$ in the two-scale sense. 
\end{enumerate}

Using the mapping $\seps$, we transform the problem in $\eqref{MicProblemEvolvingDomain}$ to the fixed domain $\oe$. Let us define 
\begin{align*}
\ueps: (0,T)\times \oe \rightarrow \R, \quad \ueps(t,x) := \tueps (t, \seps(t,x))
\end{align*}
and  for $(t,x) \in (0,T)\times \oe$
\begin{align*}
\deps(t,x)&:=\nabla \seps(t,x)^{-1} D \nabla \seps(t,x)^{-T},
\\
\qeps(t,x)&:= \nabla \seps(t,x)^{-1} \tqeps(t,\seps(t,x)) ,
\\
\veps(t,x)&:= \nabla \seps(t,x)^{-1} \partial_t \seps(t,x). 
\end{align*}
Then, by a change of coordinates
, we transform \textbf{Problem P$_T$} into \textbf{Problem P$_T$}, which is defined on the fixed domain $\oe$. We seek $\ueps$ satisfying 
\begin{align}
\begin{aligned}
\label{MicProblemFixedDomain}
\partial_t \big(\jeps\ueps\big) - \nabla \cdot \big( \epsilon^2 \jeps \deps \nabla \ueps - \epsilon \jeps \qeps \ueps + \jeps \veps \ueps \big) &= \jeps f(\ueps) &\mbox{ in }& (0,T)\times \oe,
\\
- \epsilon^2\jeps \deps \nabla \ueps \cdot \nu &= - \epsilon \jeps g(\ueps) 
  &\mbox{ on }& (0,T)\times \ge ,
\\
- \epsilon^2\jeps \deps \nabla \ueps  \cdot \nu &= 0 &\mbox{ on }& (0,T)\times \partial \Omega,
\\
\ueps(0) &= \ueps^0 &\mbox{ in }& \oe.
\end{aligned}
\end{align}

In the following, given a bounded domain $U\subset \R^n$, the duality pairing between $H^1(U)'$ and $H^1(U)$ is denoted by $\langle \cdot , \cdot \rangle_{U}$. With this we define a weak solution of Problem P$_T$ (introduced in \eqref{MicProblemFixedDomain}): 
\begin{definition}
A weak solution of the micro-problem P$_T$ is a function $\ueps \in L^2((0,T),H^1(\oe)) $, such that $\partial_t (\jeps \ueps) \in L^2((0,T),H^1(\oe)'))$, $\ueps(0) = \ueps^0$, and for every $\phi \in H^1(\oe)$ and almost every $t \in (0,T)$ it holds that
\begin{align}
\begin{aligned}
\label{VarEquaMicProblemFixedDomain}
\big\langle \partial_t (\jeps \ueps) ,& \phi \big\rangle_{\oe}+ \int_{\oe} \big[\epsilon^2 \jeps \deps \nabla \ueps - \epsilon \jeps \qeps \ueps + \jeps \veps \ueps \big) \cdot \nabla \phi dx 
\\
&= \int_{\oe} \jeps f(\ueps) \phi dx  + \epsilon \int_{\ge} \jeps g(\ueps) \phi d\sigma 
\end{aligned}
\end{align}
\end{definition}
The regularity of $\partial_t (\jeps \ueps)$ and the regularity of $\jeps$ immediately imply  $\partial_t \ueps \in L^2((0,T),H^1(\oe)')$. However, as will be seen below the estimates for the norm of $\partial_t \ueps$  are not uniform with respect to $\epsilon$, therefore we work with the product $\jeps \ueps$.
We emphasize that in the variational equation $\eqref{VarEquaMicProblemFixedDomain}
$ the time-derivative is applied to $\jeps \ueps$. This is unlike done Problem $P_W$, where the time derivative was applied to the test-function to avoid working with generalized time-derivatives for $\tueps$ defined in Bochner-spaces with values in time-dependent function spaces. We emphasise that the analysis below is done for the transformed problem \textbf{P$_T$}.
\begin{proposition}
There exists a unique weak solution of Problem $\eqref{MicProblemFixedDomain}$.
\end{proposition}
\begin{proof}
This follows e.g. by applying the Galerkin-method, and after obtaining estimates that are similar to the ones in Lemma \ref{AprioriEstimates}. We omit the details.
\end{proof}

\section{\textit{A priori} estimates}
\label{SectionAprioriEstimates}

We first introduce an $H^1(\oe)$-equivalent norm, which is more adapted to the $\epsilon$-scaling in the problem. We denote by $\he$  the space of $H^1(\oe)$ functions, equipped with the inner product
\begin{align}
\label{SPhe}
(\ueps,\weps)_{\he} := (\ueps,\weps)_{L^2(\oe)} + \epsilon^2 (\nabla \ueps , \nabla \weps )_{L^2(\oe)}. 
\end{align}
The associated norm is denoted by $\|\cdot \|_{\he}$. The spaces $H^1(\oe)$ and $\he$ are isomorphic with equivalent norms.  
%
%
More precisely, 
\begin{align*}
\Vert \peps \Vert_{\he} \le \Vert \peps \Vert_{H^1(\oe)} \le \epsilon^{-1} \Vert \peps \Vert_{\he} \quad \mbox{ for all } \peps \in H^1(\oe).
\end{align*}
 We have the embedding
\begin{align*}
L^2((0,T),\he') \hookrightarrow L^2((0,T),H^1(\oe)'),
\end{align*}
and for $F_{\epsilon} \in L^2((0,T),\he')$ it holds that
\begin{align*}
\Vert F_{\epsilon} \Vert_{L^2((0,T),H^1(\oe)')}\le \Vert F_{\epsilon} \Vert_{L^2((0,T),\he')}.
\end{align*}
Especially, for any $\weps \in L^2((0,T),\he)\cap H^1((0,T),\he')$ it holds that
\begin{align*}
\langle \partial_t \weps ,\peps \rangle_{\he',\he} = \langle \partial_t \weps ,\peps \rangle_{H^1(\oe)',H^1(\oe)} \quad \mbox{for all }\peps \in H^1(\oe).
\end{align*}
Clearly, the converse embedding also holds: if $F_{\epsilon} \in L^2((0,T),H^1(\oe)')$ then $F_{\epsilon} \in L^2((0,T),\he')$ as well. However, in this case the estimates of the operator norms in the embedding depend badly on $\epsilon$, 
\begin{align*}
\Vert F_{\epsilon} \Vert_{L^2((0,T),\he')} \le \epsilon^{-1} \Vert F_{\epsilon}\Vert_{L^2((0,T),H^1(\oe)')}.
\end{align*}
We will see in Section \ref{StrongTwoScaleResult} that the norm on $\he'$ is the appropriate norm for the time-derivative to obtain two-scale compactness results.

The following Lemma provides some basic a priori estimates for the solution $\ueps$  and the product $\jeps \ueps$.
\begin{lemma}\label{AprioriEstimates}
A constant $C >0$ not depending on $\epsilon$ exists such that for the weak solution $\ueps$ of Problem P$_T$ 
one has 
\begin{subequations}
\begin{align}
\label{AprioriEstimateUeps}
\|\ueps\|_{L^{\infty}((0,T),L^2(\oe))} +  \| \ueps \|_{L^2((0,T), \he)} &\le C,
\\
\label{AprioriEstimateTimeDerivativeJepsUeps}
\|\partial_t (\jeps \ueps) \|_{L^2((0,T),\he')} + \Vert \jeps \ueps \Vert_{L^2((0,T),\he)} &\le C.
\end{align}
\end{subequations}
\end{lemma}
\begin{proof}
We choose $\ueps $ as a test function in equation $\eqref{VarEquaMicProblemFixedDomain}$. The positivity  of $D$ and the properties of $\seps$ imply the coercivity of $\deps$. Using $c_0\le \jeps$ and the formulas (remember that $\partial_t \jeps$ is a $C^0$-function)
\begin{align*}
\big\langle \partial_t (\jeps \ueps), \ueps \big\rangle_{\oe} = \frac{1}{2 } \frac{d}{dt} \|\sqrt{\jeps} \ueps\|^2_{L^2(\oe)} + \frac12 \int_{\oe} \partial_t \jeps \ueps^2 dx,
\end{align*}
and (since $\partial_t \jeps = \nabla \cdot (\jeps \veps)$, what is an easy consequence of \cite[p. 117]{MarsdenHughes})
\begin{align*}
\int_{\oe} \jeps \ueps \veps \cdot \nabla \ueps dx = -\frac12 \int_{\oe} \partial_t \jeps \ueps^2 dx + \frac12 \int_{\partial \oe} \jeps \veps \cdot \nu \ueps^2 d\sigma,
\end{align*}
we obtain for a constant $d>0$
\begin{align*}
\frac12 \frac{d}{dt} &\|\sqrt{\jeps} \ueps \|^2_{L^2(\oe)} +  d \epsilon^2 \|\nabla \ueps \|^2_{L^2(\oe)}
\\
\le& -   \frac12 \int_{\oe} \partial_t \jeps \ueps^2 dx + \epsilon \int_{\oe} \jeps \ueps \qeps \cdot \nabla \ueps dx - \int_{\oe} \jeps \ueps \veps \cdot \nabla \ueps dx 
\\
 &+ \int_{\oe} \jeps f(\ueps) \ueps dx + \epsilon \int_{\ge} \jeps g(\ueps ) \ueps d\sigma  
\\
=&  \epsilon \int_{\oe} \jeps \ueps \qeps \cdot \nabla \ueps dx + \int_{\oe} \jeps f(\ueps) \ueps dx 
\\
&+ \epsilon \int_{\ge} \jeps g(\ueps) \ueps d\sigma - \frac12 \int_{\partial \oe} \jeps \veps \cdot \nu \ueps^2 d\sigma =:\sum_{j=1}^4
 A_{\epsilon}^{(j)}.
\end{align*}
We continue with the terms on the right-hand side. 
For any $\theta >0$, a $C(\theta)>0$ exists such that for the first term one has 
\begin{align*}
A_{\epsilon}^{(1)} \le C \epsilon \|\jeps \qeps\|_{L^{\infty}(\oe)} \|\ueps\|_{L^2(\oe)} \|\nabla \ueps \|_{L^2(\oe)} \le C(\theta)\|\ueps\|^2_{L^2(\oe)} + \theta\epsilon^2 \|\nabla \ueps\|^2_{L^2(\oe)} .
\end{align*}
Concerning the nonlinear terms, we only consider the boundary integral. All others can be treated in a similar way. We use the trace inequality for periodically perforated domains, stating that for every $\theta>0$, there exists $C(\theta)>0$, such that for all $\weps \in H^1(\oe)$ one has 
\begin{align}\label{TraceInequality}
\epsilon\|\weps\|_{L^2(\partial \oe )}^2 \le C(\theta)\|\weps\|^2_{L^2(\oe)} + \theta \epsilon ^2 \|\nabla \ueps\|^2_{L^2(\oe)}.
\end{align}
Together with the Lipschitz continuity of $g$, this implies that 
\begin{align*}
A_{\epsilon}^{(3)} &\le C\epsilon \int_{\ge} 1 + |\ueps|^2 d\sigma \le C\big(1 + \epsilon\|\ueps\|_{L^2(\ge)}^2\big)
\\
&\le C(\theta) \big(1 + \|\ueps\|^2_{L^2(\oe)} \big) + \theta \epsilon^2 \|\nabla \ueps\|^2_{L^2(\oe)} .
\end{align*}
For $A_{\epsilon}^{(4)}$, we use similar arguments as above and Assumption \ref{AssumptionsAprioriSeps} to obtain 
\begin{align*}
A_{\epsilon}^{(4)} &\le C \epsilon \|\ueps\|_{L^2(\partial \oe)}^2 \le C(\theta) \|\ueps\|^2_{L^2(\oe)} + \theta \epsilon^2 \|\nabla \ueps\|^2_{L^2(\oe)}.
\end{align*}

Choosing $\theta $ small enough, the terms on the right including the gradients can be absorbed by similar terms on the left. Using the Gronwall-inequality and the $\he$-norm induced by the inner product $\eqref{SPhe}$, we obtain inequality $\eqref{AprioriEstimateUeps}$. 

The $L^2((0,T),\he)$-estimate for $\jeps \ueps$ follows directly from $\eqref{AprioriEstimateUeps}$ and the properties of $\jeps$. To prove the inequality for the time-derivative $\partial_t (\jeps \ueps)$, we choose $\phi \in \he$ with $\|\phi\|_{\he} \le 1$ as a test function in  $\eqref{VarEquaMicProblemFixedDomain}$, and apply similar arguments as above.
\end{proof}

We emphasize that in the proof above we used the $C^1$-regularity of $\jeps$, but we have not assumed a uniform bound with respect to $\epsilon$ for the $C^0$-norm of the time-derivative $\partial_t \jeps$. We overcome this problem by using the equality $\partial_t \jeps = \nabla \cdot (\jeps \veps)$.

\begin{remark}\label{BemerkungAprioriEstimaetZeitableitungUeps}
Due to the $C^1$-regularity of $\jeps$ and the product rule we have 
\begin{align*}
\partial_t \ueps = \jeps^{-1} \partial_t (\jeps \ueps) - \jeps^{-1} \ueps \partial_t \jeps \in L^2((0,T),\he') \hookrightarrow L^2((0,T),H^1(\oe)').
\end{align*}
However, we only have a uniform bound for the $L^2((0,T),\he')$-norm of $\partial_t \jeps$, see \ref{AssumptionBoundsJeps}. Therefore, we obtain no bounds for $\partial_t \ueps$ in the  $L^2((0,T),\he')$ norm (or in $L^2((0,T),H^1(\oe)')$) that are uniform in $\epsilon$. Under the additional assumption
\begin{align*}
\Vert \partial_t \jeps \Vert_{L^{\infty}((0,T)\times \oe)} \le C,
\end{align*}
from the results above one also gets 
\begin{align*}
\Vert \partial_t \ueps \Vert_{L^2((0,T),\he')} \le C.
\end{align*}
\end{remark}

\subsection{Estimates for the shifted functions}

To prove strong compactness results which are needed to pass to the limit in the nonlinear terms, we need additional a priori estimates for the difference between shifted functions and the function itself. 
For $h>0$ let us define the set $\Omega^h:= \{x\in \Omega \, : \, \mathrm{dist}(x,\partial \Omega) > h \}$. Further we define 
\begin{align*}
K_{\epsilon}^h&:= \{k\in \Z^n\, : \, \epsilon(Y + k) \subset \Omega^h\},
\\
\oeh&:= \mathrm{int}\left(\bigcup_{k \in K_{\epsilon}^h} \epsilon\big(\overline{Y^{\ast}}+ k \big)\right),
\\
\geh&:= \mathrm{int} \left(\bigcup_{k \in K_{\epsilon}^h }\epsilon \big(\overline{\Gamma} + k \big) \right).
\end{align*}
Then, for $\ell \in \Z^n $ with $|\ell\epsilon|< h$ we define for an arbitrary function $\veps : \oe \rightarrow \R$ the shifted function
\begin{align*}
\veps^\ell(x):= \veps(x + \ell\epsilon),
\end{align*} 
and the difference between the shifted function and the function itself
\begin{align*}
\delta \veps(x):= \veps^\ell(x) - \veps(x) = \veps (x+\ell\epsilon) - \veps(x).
\end{align*}
In the following, we want to estimate $\delta (\jeps \ueps)$. Therefore, we need the variational equation for the shifted function $\ueps^\ell$, respectively $\jeps^\ell \ueps^\ell$. Due to the low regularity for the time-derivative $\partial_t (\jeps \ueps) \in L^2((0,T),H^1(\oe)')$, it is not obvious how to obtain the time-derivative of $(\jeps \ueps)^\ell$, since this function is only defined on $\oeh$. Therefore, we argue in the following way: First of all, we define the space 
\begin{align*}
\hoeh:= \left\{ \phi_{\epsilon} \in H^1(\oeh)\, : \, \phi_{\epsilon} = 0 \mbox{ on } \partial \oeh \setminus \geh \right\}
\end{align*}
of Sobolev functions with zero trace on the outer boundary of $\oeh$.
For $\veps \in H^1((0,T),H^1(\oe)')$ it holds that the restriction to $\oeh$  fulfills $\veps \in H^1((0,T),\hoeh')$ with 
\begin{align*}
\langle \partial_t \veps , \phi \rangle_{\hoeh' , \hoeh} = \langle \partial_t \veps , \tphi \rangle_{H^1(\oe)',H^1(\oe)},
\end{align*}
for all $\phi \in \hoeh$ and $\tphi$ denotes the zero extension of $\phi$ to $\oe$. Further, we have
\begin{align*}
\langle \partial_t \veps^\ell , \phi \rangle_{\hoeh',\hoeh} = \langle \partial_t \veps , \tphi^{-l}\rangle_{H^1(\oe)',H^1(\oe)}.
\end{align*}
Finally, the space $\he^h$ is defined in the same way as $\he$, with $\oeh$ replacing $\oe$.

\begin{lemma}\label{ApriorEstimatesShifts}
Let $\ueps$ be the sequence of weak solutions of Problem P$_T$ (stated in $\eqref{MicProblemFixedDomain}$), $0 < h \ll 1$ and $\ell\in \Z^n$ such that $|\ell\epsilon|<h$. Then it holds
\begin{align*}
\Vert \delta(\jeps \ueps)\Vert_{L^2((0,T),\he^{2h})}
\le& C \|\delta (\jeps(0)\ueps^0)\|_{L^2(\oeh)} + C|\ell\epsilon| + C\sqrt{\epsilon},
\end{align*}
with a constant $C > 0$ which may depend on $h$.
\end{lemma}
\begin{proof}
For the ease of writing we use the notation  
\begin{align*}
\weps(t,x):= \jeps(t,x) \ueps(t,x).
\end{align*}
Using the product rule gives 
\begin{align*}
\jeps \deps \nabla \ueps = \deps \nabla \weps - \weps \frac{\deps}{\jeps} \nabla \jeps.
\end{align*}
An elemental calculation gives us the following variational equation for all  $\phi \in \hoeh$ and almost everywhere in $(0,T)$
\begin{align}
\begin{aligned}
\label{AuxiliaryEquationEstimatesShifts}
\big\langle& \partial_t \delta \weps , \phi \big\rangle_{\hoeh',\hoeh} + \epsilon^2 \int_{\oeh} \deps \nabla \delta \weps \cdot \nabla \phi dx
\\
=& \langle \partial_t \weps^\ell , \phi \rangle_{\hoeh',\hoeh} + \epsilon^2 \int_{\oeh} \deps^\ell \nabla \weps^\ell \cdot \nabla \phi dx 
 - \epsilon^2 \int_{\oeh} \delta \deps \nabla \weps^\ell \cdot \nabla \phi dx 
\\
&- \langle \partial_t \weps , \phi \rangle_{\hoeh',\hoeh} - \epsilon^2 \int_{\oeh} \deps \nabla \weps \cdot \nabla \phi dx 
\\
=& \langle \partial_t \weps^\ell , \phi \rangle_{\hoeh',\hoeh} + \epsilon^2 \int_{\oeh} \jeps^\ell \deps^\ell \nabla \ueps^\ell \cdot \nabla \phi dx 
\\
&- \langle \partial_t \weps , \phi \rangle_{\hoeh',\hoeh} - \epsilon^2 \int_{\oeh} \jeps \deps \nabla \ueps \cdot \nabla\phi dx 
\\
&- \epsilon^2 \int_{\oeh} \delta \deps \nabla \weps^\ell \cdot \nabla \phi dx + \epsilon^2 \int_{\oeh} \delta\left(\weps \frac{\deps}{\jeps} \nabla \jeps \right) \cdot \nabla \phi dx
\\
=& - \epsilon^2 \int_{\oeh} \delta \deps \nabla \weps^\ell \cdot \nabla \phi dx + \epsilon^2 \int_{\oeh} \delta \weps \frac{\deps}{\jeps} \nabla \jeps \cdot \nabla \phi dx 
\\
&+ \epsilon^2 \int_{\oeh} \delta \left(\frac{\deps}{\jeps} \nabla \jeps \right)  \weps^\ell \cdot \nabla \phi dx + \epsilon \int_{\oeh} \weps^\ell \delta \qeps \cdot \nabla \phi dx 
\\
&+ \epsilon \int_{\oeh} \delta \weps \qeps \cdot \nabla \phi dx -  \int_{\oeh} \weps^\ell \delta \veps \cdot \nabla \phi dx 
\\
&- \int_{\oeh} \delta \weps \veps \cdot \nabla \phi dx + \int_{\oeh} \big(\jeps^\ell f(\ueps^\ell) - \jeps f(\ueps) \big) \phi dx 
\\
&+ \epsilon \int_{\geh} \big(\jeps^\ell g(\ueps^\ell) - \jeps g(\ueps) \big) \phi d\sigma
 =: \sum_{j=1}^{9} A_{\epsilon}^{(j)}.
\end{aligned}
\end{align}
Now, we choose $\phi = \eta^2 \delta \weps $, where $\eta \in C^{\infty}\big(\overline{\oeh}\big)$ is a cut-off function with $0 \le \eta \le 1$, $\eta =1 $ in $\overline{\oeth}$ and $\eta = 0$ in $\oe\setminus \oeh$.
This implies 
\begin{align*}
\big\langle \partial_t \delta \weps , \eta^2 \delta \weps  \big\rangle_{\hoeh',\hoeh} 
&= \frac12 \frac{d}{dt} \|\eta \delta \weps \|^2_{L^2(\oe)}
\end{align*}
and  the coercivity of $\deps$ implies the existence of a constant $d>0$ such that (we extend every function by zero outside $\oe$)
\begin{align*}
\epsilon^2 \int_{\oeh} \deps &\nabla \delta \weps \cdot \nabla \big(\eta^2 \delta \weps\big) dx 
\\
&= \epsilon^2 \int_{\oe} \deps \eta^2 \nabla \delta \weps \cdot \nabla \delta \weps dx + 2\epsilon^2 \int_{\oe} \eta \delta \weps \deps \nabla \delta \weps \cdot \nabla \eta dx 
\\
&\geq \epsilon^2 d \|\eta \nabla \delta \weps \|^2_{L^2(\oe)} + 2 \epsilon^2 \int_{\oe} \eta \delta \weps \deps \nabla \delta \weps \cdot \nabla \eta dx.
\end{align*}
For the last term we use Lemma \ref{AprioriEstimates} to obtain for every $\theta >0$ a constant $C(\theta)>0$ such that
\begin{align*}
\epsilon^2 \left| \int_{\oe} \eta \delta \weps \deps \nabla \delta \weps \cdot \nabla \eta dx \right| &\le C(\theta) \epsilon^2 \|\delta \weps \|^2_{L^2(\oeh)} + \theta \epsilon^2 \|\eta \nabla \delta  \weps \|^2_{L^2(\oe)} 
\\
&\le C(\theta) \epsilon^2 + \theta \epsilon^2 \|\eta \nabla \delta \weps \|^2_{L^2(\oe)}.
\end{align*}
Now we estimate the terms $A_{\epsilon}^{(j)}$. From the assumptions on $\seps$  we obtain $\|\delta \deps \|_{L^{\infty}((0,T)\times \oeh)} \le C|\ell\epsilon|$. This, together with the a priori estimates in Lemma \ref{AprioriEstimates} and after integration with respect to time, implies that for arbitrary $\theta >0$, a constant $C(\theta)>0$ exists s.t.
\begin{align*}
\int_0^t  A_{\epsilon}^{(1)} &ds =
-\epsilon^2 \int_0^t\int_{\oeh} \eta^2 \delta \deps \nabla \weps^\ell \cdot \nabla \delta \weps  + 2 \eta \delta \deps \delta \weps \nabla \weps^\ell \cdot \nabla \eta dxds
\\
\le& C\epsilon^2 \| \delta \deps \|_{L^{\infty}((0,t)\times \oeh)} \|\nabla \weps^\ell \|_{L^2((0,t)\times \oeh)} \|\eta \nabla \delta \weps\|_{L^2((0,t)\times \oeh)}
\\
&+ C \epsilon^2 \|\eta \delta \weps\|_{L^2((0,t)\times \oeh)} \|\delta \deps \|_{L^{\infty}((0,t)\times \oeh)} \|\nabla \weps^\ell\|_{L^2((0,t)\times \oeh)}
\\
\le& C(\theta) |\ell\epsilon|^2   + \theta \epsilon^2 \|\eta \nabla \delta \weps\|_{L^2((0,t)\times \oeh)}^2 + C \|\eta \delta \weps\|^2_{L^2((0,t)\times \oeh)} + C\epsilon^2|\ell\epsilon|^2.
\end{align*}
For the term $A_{\epsilon}^{(2)}$ we use Assumptions (A2) and (A3) to conclude that  
\begin{align*}
\left\|\frac{\deps}{\jeps} \nabla \jeps\right\|_{L^{\infty}((0,T)\times \oeh)} \le \frac{C}{\epsilon}.  
\end{align*}
Using the above gives 
\begin{align*}
\int_0^t A_{\epsilon}^{(2)} & ds =
2\epsilon^2 \int_0^t\int_{\oeh} \delta \weps^2 \frac{\deps}{\jeps}\nabla \jeps \cdot \nabla \eta \eta  +  \delta \weps \frac{\deps}{\jeps}\nabla \jeps \cdot \nabla \delta \weps \eta^2 dxds
\\
&\le C \epsilon \|\delta \weps \|^2_{L^2((0,t)\times \oeh)} + C \epsilon \|\eta \delta \weps \|_{L^2((0,t)\times \oeh)} \|\eta \nabla \delta \weps \|_{L^2((0,t)\times \oeh)}
\\
&\le C\epsilon + C(\theta) \|\eta \delta \weps \|^2_{L^2((0,t)\times \oeh)} + \theta \epsilon^2 \|\eta \nabla \delta \weps \|^2_{L^2((0,t)\times \oeh)}. 
\end{align*}
For $A_{\epsilon}^{(3)}$, we use Assumption \ref{AssumptionEstimateShiftGradientJeps} to obtain
\begin{align*}
\left\|\delta \left(\frac{\deps}{\jeps}\nabla \jeps\right)\right\|_{L^{\infty}((0,T)\times \oeh)} \le C|l|.
\end{align*}
This implies that 
\begin{align*}
\int_0^t A_{\epsilon}^{(3)} ds \le& C |\ell\epsilon| \epsilon \| \weps^\ell \|_{L^2((0,t)\times \oeh)} \|\eta \nabla \delta \weps \|_{L^2((0,t)\times \oeh)} 
\\
&+ C|\ell\epsilon| \epsilon\|\eta \delta \weps \|_{L^2((0,t)\times \oeh)} \| \weps^\ell\|_{L^2((0,t)\times \oeh)}
\\
\le& C(\theta) |\ell\epsilon|^2 + \theta \epsilon^2 \|\eta \nabla \delta \weps \|_{L^2((0,t)\times \oeh)}^2 + C \|\eta \delta \weps \|^2_{L^2((0,t)\times \oeh)} + C|\ell\epsilon|^2.
\end{align*}
Using the assumptions on $\tqeps$, we obtain
\begin{align*}
\int_0^t A_{\epsilon}^{(4)} ds\le& C \epsilon \|\weps^\ell \|_{L^2((0,t)\times \oeh)} \|\delta \qeps \|_{L^{\infty}((0,t)\times \oeh)} \|\eta \nabla \weps \|_{L^2(\oeh)} 
\\
&+ C \epsilon \|\weps^\ell \|_{L^2((0,t)\times \oeh)}\|\eta \delta \weps \|_{L^2((0,t)\times \oeh)} \|\delta \qeps\|_{L^{\infty}((0,t)\times \oeh)}
\\
\le& C(\theta) |\ell\epsilon|^2 + \theta \epsilon^2 \|\eta \nabla \delta \weps \|^2_{L^2((0,t)\times \oeh)} + C\|\eta \delta \weps \|_{L^2((0,t)\times \oeh)}^2 + C \epsilon^2 |\ell\epsilon|^2.
\end{align*}
For the fifth term, employing similar arguments as above leads to 
\begin{align*}
\int_0^t A_{\epsilon}^{(5)} ds \le C(\theta) \|\eta \delta \weps\|^2_{L^2((0,t)\times \oeh)} + \theta \epsilon^2 \|\eta \nabla \delta \weps\|^2_{L^2((0,t)\times \oeh)}+ C\epsilon.
\end{align*}
For $A_{\epsilon}^{(6)}$, we use Assumption \ref{AssumptionEstimateShiftVelocitySurface} to obtain 
\begin{align*}
\int_0^t A_{\epsilon}^{(6)} ds \le& \|\delta \veps \|_{L^{\infty}((0,t)\times \oeh)} \|\weps^\ell\|_{L^2((0,t)\times \oeh)} \|\eta \nabla \delta \weps\|_{L^2((0,t)\times \oeh)} 
\\
&+ C\|\weps^\ell\|_{L^2((0,t)\times \oeh)} \|\eta \delta \weps\|_{L^2((0,t)\times \oeh)} \|\delta \veps\|_{L^{\infty}((0,t)\times \oeh)}
\\
\le& C(\theta) |\ell\epsilon|^2 + \theta \epsilon^2 \|\eta \nabla \delta \weps \|^2_{L^2((0,t)\times \oeh)} + C \|\eta \delta \weps \|_{L^2((0,t)\times \oeh)}^2 + C \epsilon^2 |\ell\epsilon|^2, 
\end{align*}
while for $A_{\epsilon}^{(7)}$ we get
\begin{align*}
\int_0^t A_{\epsilon}^{(7)} ds \le C(\theta) \|\eta \delta \weps \|_{L^2((0,t)\times \oeh)}^2 + \theta \epsilon^2 \|\eta \nabla \delta \weps \|^2_{L^2((0,t)\times \oeh)} + C\epsilon.
\end{align*}
From the nonlinear terms, we only consider the boundary term $A_{\epsilon}^{(9)}$, since the other one can be treated in a similar way. One has 
\begin{align*}
\int_0^t A_{\epsilon}^{(9)} ds =& \epsilon \int_0^t\int_{\geh} \delta \jeps g(\ueps^\ell) \eta^2 \delta \weps d\sigma ds + \epsilon \int_0^t \int_{\geh} \jeps \big(g(\ueps^\ell) - g(\ueps) \big) \eta^2 \delta \weps d\sigma ds
\\
=&: B_{\epsilon}^{(1)} + B_{\epsilon}^{(2)}
\end{align*}
Using Assumption \ref{AssumptionEstimateShiftGradientJeps}, Lemma \ref{AprioriEstimates}, and the trace inequality $\eqref{TraceInequality}$, for $B_{\epsilon}^{(1)}$ we get  
\begin{align*}
B_{\epsilon}^1 \le&
C \epsilon^{\frac12} |\ell\epsilon| \|\eta \delta \weps\|_{L^2((0,t)\times \geh)} + C \epsilon |\ell\epsilon| \|\ueps\|_{L^2((0,t)\times \geh)} \|\eta \delta \weps \|_{L^2((0,t)\times \geh)}
\\
\le& C \epsilon^{\frac12} |\ell\epsilon| \|\eta \delta \weps\|_{L^2((0,t)\times \geh)}
\\
\le& C(\theta)  |\ell\epsilon| \|\eta \delta \weps \|_{L^2((0,t)\times \oeh)} + \theta  \epsilon |\ell\epsilon| \|\nabla (\eta \delta \weps )\|_{L^2((0,t)\times \oeh)}
\\
\le& C(\theta) \| \eta \delta \weps\|_{L^2((0,t)\times \oeh)}^2+ C(\theta)|\ell\epsilon|^2 + C\epsilon |\ell\epsilon| + \theta \epsilon^2 \|\eta \nabla \delta \weps\|_{L^2((0,t)\times \oeh)}^2.
\end{align*}
Using the positivity of $\jeps$ and similar arguments as for $B_{\epsilon}^{(1)}$ gives 
\begin{align*}
B_{\epsilon}^{(2)} \le& C \epsilon \int_0^t \int_{\geh} |\jeps \ueps^\ell - \jeps \ueps | \eta^2 |\delta \weps | d\sigma ds
\\
\le& C\epsilon \int_0^t \int_{\geh} |\delta \jeps| |\ueps^\ell| \eta^2 |\delta \weps| + \eta^2|\delta \weps|^2 d\sigma ds
\\
\le& C|\ell\epsilon| \epsilon \|\ueps^\ell\|_{L^2((0,t)\times \geh)} \|\eta \delta \weps \|_{L^2((0,t)\times \geh)} + C\epsilon \|\eta \delta \weps\|^2_{L^2((0,t)\times \geh)} 
\\
\le& C(\theta) \|\eta \delta \weps\|^2_{L^2((0,t)\times \oeh)} + C(\theta)|\ell\epsilon|^2 + C\epsilon |\ell\epsilon| + \theta \epsilon^2 \|\eta \nabla \delta \weps\|_{L^2((0,t)\times \oeh)} + C\epsilon^2.
\end{align*}
Integrating $\eqref{AuxiliaryEquationEstimatesShifts}$ with respect to time, applying the above estimates and the Gronwall-inequality, after choosing $\theta$ small enough, gives the desired result.
\end{proof}

\section{Two-scale compactness results}
\label{StrongTwoScaleResult}

In this section  we prove a general strong two-scale compactness result for an arbitrary sequence based on \textit{a priori} estimates and estimates for the differences of the shifts. We make use of the unfolding operator in perforated domains, see \cite{CioranescuDamlamianDonatoGrisoZakiUnfolding}. First of all, let us repeat the definition of two-scale convergence, which was introduced in \cite{Nguetseng} and \cite{Allaire_TwoScaleKonvergenz}. \\
\textbf{Notation:} To respect the notations used in early papers wherein the following concepts have been introduced, in this Section $\ueps$ and $\veps$ will denote arbitrary functions, and are not related to the microscopic model discussed here.

\begin{definition}\label{DefinitionTSConvergence}
A sequence $\ueps \in  L^p((0,T)\times \Omega)$ for $p\in [1,\infty)$ is said to converge  in the two-scale sense to the limit function $u_0\in L^p((0,T)\times  \Omega \times Y)$, if for every $\phi \in L^{p'}((0,T)\times \Omega, C_{\per}(\overline{Y}))$  the following relation holds
\begin{align*}
\lim_{\epsilon\to 0}\int_0^T \int_{\Omega}u_{\epsilon}(t,x)\phi\left(t,x,\frac{x}{\epsilon}\right)dxdt = \int_0^T\int_{\Omega}\int_Y u_0(t,x,y)\phi(t,x,y)dydxdt .
\end{align*}
A  two-scale convergent sequence $u_{\epsilon}$ convergences strongly in the two-scale sense to $u_0$, if  
\begin{align*}
\lim_{\epsilon\to 0}\|u_{\epsilon}\|_{L^p((0,T)\times \Omega)} =\|u_0\|_{L^p((0,T)\times \Omega  \times Y)} .
\end{align*}
\end{definition} 

\begin{remark}\mbox{}
\begin{enumerate}[label = (\roman*)]
\item By replacing $Y$ with $Y^{\ast}$ and $\Omega$ with $\oe$, the Definition \ref{DefinitionTSConvergence} can be easily extended to sequences in $L^p((0,T)\times \oe)$. In both cases we will use the same notation for the two-scale convergence.
\item  The two scale convergence introduced above should actually be referred to as "weak two scale convergence". However, in accordance with the definition in \cite{Allaire_TwoScaleKonvergenz} we neglect the word "weak" and only use "strong" to highlight the "strong two-scale convergence".
\end{enumerate}

\end{remark}

In \cite{AllaireDamlamianHornung_TwoScaleBoundary,Neuss_TwoScaleBoundary} the method of two-scale convergence was extended to oscillating surfaces:
\begin{definition}
A sequence of functions $\ueps  \in L^p((0,T)\times\Gamma_{\epsilon})$ for $p \in [1,\infty)$ is said to converge  in the two-scale sense on the surface $\Gamma_{\epsilon}$ to a limit $u_0\in L^p((0,T)\times \Omega \times \Gamma)$, if for every $\phi \in C\left([0,T]\times\overline{\Omega},C_{per}(\Gamma)\right)$ it holds that
\begin{align*}
\lim_{\epsilon \to 0} \epsilon \int_0^T\int_{\Gamma_{\epsilon}} u_{\epsilon}(t,x)\phi\left(t,x,\frac{x}{\epsilon}\right)d\sigma dt = \int_0^T\int_{\Omega}\int_{\Gamma} u_0(t,x,y)\phi(t,x,y)d\sigma_ydxdt .
\end{align*}

We say a  two-scale convergent sequence $u_{\epsilon}$ converges strongly in the two-scale sense, if additionally it holds that
\begin{align*}
\lim_{\epsilon\to 0}  \epsilon^{\frac{1}{p}}\|u_{\epsilon}\|_{L^p((0,T)\times \Gamma_{\epsilon})} = \|u_0\|_{L^p((0,T)\times \Omega \times \Gamma)} .
\end{align*}
\end{definition}
An equivalent characterization of the two-scale convergence can be given via the unfolding operator, see \cite{CioranescuDamlamianDonatoGrisoZakiUnfolding}, which is defined by ($p \in [1,\infty)$) for perforated domains by 
\begin{align*}
\te : L^p((0,T)\times \oe) \rightarrow L^p((0,T)\times \Omega \times Y^{\ast}), \quad \te \ueps(t,x,y)= \ueps\left(t,\epsilon \left[\fxe\right] + \epsilon y \right),
\end{align*}
where $[\cdot]$ denotes the integer part. In the same way we can define the unfolding operator on the whole domin $\Omega$ and we use the same notation. We also define the boundary unfolding operator by
\begin{align*}
\tbe: L^p((0,T)\times \ge) \rightarrow L^p((0,T)\times \Omega \times \Gamma) , \quad \tbe \ueps (t,x,y) = \ueps \left(t,\epsilon \left[\fxe\right] + \epsilon y \right).
\end{align*}
Let us summarize some important properties of the unfolding operator, which can be found in \cite{CioranescuDamlamianDonatoGrisoZakiUnfolding}
\begin{enumerate}
[label = (\roman*)]
\item For $\ueps , \veps \in L^2((0,T)\times \oe)$ it holds that
\begin{align*}
(\ueps,\veps)_{L^2((0,T)\times \oe)} = (\te \ueps , \te \veps)_{L^2((0,T)\times \Omega \times Y^{\ast})}.
\end{align*}
\item For $\ueps \in L^2((0,T),H^1(\oe))$ we have $\nabla_y \te \ueps = \epsilon \te (\nabla \ueps)$.
\item For $\ueps ,\veps \in L^2((0,T)\times \ge)$ it holds that
\begin{align*}
(\ueps,\veps)_{L^2((0,T)\times \ge)} = \epsilon^{-1} (\tbe \ueps , \tbe \veps )_{L^2((0,T)\times \Omega \times \Gamma)}.
\end{align*}
\end{enumerate}
The statements above remain valid if we replace $\oe$ by $\Omega$.

In \cite{BourgeatLuckhausMikelic} the following relation between the unfolding operator and the two-scale convergence has been proved (see also \cite{lenczner1997homogeneisation}).
\begin{lemma}
For every bounded sequence $u_{\epsilon} \in L^p((0,T)\times \Omega)$ the following statements are equivalent
\begin{itemize}
\item[(i)] $u_{\epsilon}\rightarrow u_0$  weakly/strongly in the two-scale sense for $u_0 \in L^p((0,T)\times \Omega \times Y)$.
\item[(ii)] $\te u_{\epsilon}\rightarrow u_0$ weakly/strongly in $L^p((0,T)\times \Omega \times Y)$.
\end{itemize}
The same result holds for $\ueps $ in $L^p((0,T)\times \ge)$ with $\epsilon^{\frac{1}{p}}\|u_{\epsilon}\|_{L^p((0,T)\times \ge)}\le C$ and the boundary unfolding operator $\tbe $. The result is also true, if we replace $\Omega$ with $\oe$ and $Y$ with $Y^{\ast}$.
\end{lemma}

Now, we define the averaging operator $\ue : L^2((0,T)\times \Omega \times Y^{\ast})\rightarrow L^2((0,T)\times \oe)$ as the $L^2$-adjoint of the unfolding operator. An elemental calculation, see \cite{CioranescuDamlamianDonatoGrisoZakiUnfolding}, shows
\begin{align*}
\ue (\phi)  (t,x) = \int_{Y^{\ast}} \phi \left(t,\epsilon \left(y + \left[\fxe\right]\right), \left\{\fxe\right\}\right)dy,
\end{align*}
with $\{z\} = z - [z]$.
As an immediate consequence of the $L^2$-adjointness we obtain
\begin{align}\label{EstimateAveragingOperator}
\|\ue(\phi)\|_{L^2((0,T)\times \oe)} \le \|\phi\|_{L^2((0,T)\times \Omega \times Y^{\ast})}.
\end{align}
Further, let us define the space
\begin{align*}
\hoy := \left\{ \phi \in H^1(Y^{\ast}) \, : \, \phi = 0 \mbox{ on } \partial Y \cap \partial Y^{\ast} \right\}.
\end{align*}
Then we have the following result
\begin{proposition}\label{DerivativeAveragingOperator}
For all $\phi \in L^2((0,T)\times \Omega ,\hoy)$ it holds that
$\epsilon \nabla_x \ue(\phi) = \ue(\nabla_y \phi )$.
 Especially, we have
\begin{align*}
\|\ue(\phi)\|_{L^2((0,T),\he)} \le \|\phi\|_{L^2((0,T)\times \Omega,\hoy)}.
\end{align*}
In other words, the restriction of $\ue$ to $ L^2((0,T)\times \Omega ,\hoy)$ fulfills
\begin{align*}
\ue: L^2((0,T)\times \Omega , \hoy) \rightarrow L^2((0,T),\he).
\end{align*}
\end{proposition}
\begin{proof}
This is an easy consequence of the $L^2$-adjointness of $\ue$ and $\nabla_y \te = \epsilon\te \nabla_x$. In fact for $\phi \in L^2((0,T)\times \Omega , \hoy ) $ and $\psi \in C_0^{\infty}(\oe)$ we have almost everywhere in $(0,T)$
\begin{align*}
\int_{\oe} \ue (\phi) \partial_{x_i} \psi dx &= \int_{\Omega}\int_{Y^{\ast}} \phi \te (\partial_{x_i} \psi) dy dx 
\\
&= \frac{1}{\epsilon} \int_{\Omega}\int_{Y^{\ast}} \phi \partial_{y_i} \te (\psi)dy dx 
\\
&= - \frac{1}{\epsilon} \int_{\Omega}\int_{Y^{\ast}} \partial_{y_i} \phi \te (\psi) dy dx + \frac{1}{\epsilon} \int_{\Omega} \int_{\partial Y^{\ast}} \phi \te (\psi) \nu_i d\sigma_y dx
\\
&= - \frac{1}{\epsilon} \int_{\oe} \ue (\partial_{y_i} \phi ) \psi dx,
\end{align*}
where the boundary term vanishes since $\te(\psi) = 0$ on $\Gamma$ and $\phi = 0$ on $\partial Y \cap \partial Y^{\ast}$. The inequality is an easy consequence of $\eqref{EstimateAveragingOperator}$ and the definition of $\he$.
\end{proof}

This result gives a formula for the generalized time-derivative of $\te \ueps$.
\begin{proposition}\label{DerivativeUnfoldingOperator}
Let $\veps \in L^2((0,T)\times \oe) \cap H^1((0,T),\he')$. Then we have 
\begin{align*}
\te \veps \in H^1((0,T),L^2(\Omega,\hoy)')
\end{align*}
with
\begin{align*}
\langle \partial_t \te \veps , \phi \rangle_{L^2(\Omega,\hoy)',L^2(\Omega,\hoy)} = \langle \partial_t \veps , \ue (\phi) \rangle_{\oe}.
\end{align*}
Additionally, we have
\begin{align*}
\|\partial_t \te \veps \|_{L^2((0,T),L^2(\Omega,\hoy)')} \le \|\partial_t \veps\|_{L^2((0,T),\he')}.
\end{align*}
\end{proposition}
\begin{proof}
For $\phi \in L^2(\Omega,\hoy)$ and $\psi \in \mathcal{D}(0,T)$ we obtain immediately from Proposition \ref{DerivativeAveragingOperator}
\begin{align*}
\int_0^T &\int_{\Omega}\int_{Y^{\ast}} \te \veps (t,x,y) \phi(x,y)  \psi'(t) dy dx dt
\\
&= \int_0^T \int_{\oe} \veps(t,x) \ue(\phi)(x) \psi'(t) dx dt
= -\int_0^T \langle \partial_t \veps (t) , \ue (\phi)\rangle_{\oe} \psi(t) dt.
\end{align*}
The inequality is a direct consequence of $\langle \partial_t \veps , \ue (\phi)\rangle_{\oe} = \langle \partial_t \veps , \ue (\phi)\rangle_{\he',\he}$ and  Proposition \ref{DerivativeAveragingOperator} .
\end{proof}

In the next proposition we show that under suitable a priori estimates, the existence of a generalized time-derivative  for a two-scale convergent sequence extends to the limit function. First, we introduce the difference quotient with respect to time. Given a Banach space $X$ and with $0< h \ll 1$, for an arbitrary function $v \in L^2((0,T),X)$ we define 
\begin{align*}
\partial_t^h v(t,x):= \frac{v(t+h,x) - v(t,x)}{h} \quad \mbox{ for } t \in (0,T-h), \, x \in X.
\end{align*}

\begin{proposition}\label{TimeDerivativeTSConvergence}
Let $\veps \in L^2((0,T),\he)\cap H^1((0,T),\he')$ with
\begin{align*}
\|\veps \|_{L^2((0,T),\he)} + \|\partial_t \veps \|_{L^2((0,T),\he')} \le C.
\end{align*}
Then there exists $v_0 \in L^2((0,T)\times \Omega, H_{\per}^1(Y^{\ast})) \cap H^1((0,T),L^2(\Omega,H^1_{\per}(Y^{\ast})'))$, 
such that up to a subsequence it holds that
\begin{align*}
\veps &\rightarrow v_0 &\mbox{ in the two-scale sense,}
\\
\epsilon \nabla \veps &\rightarrow \nabla_y v_0 &\mbox{ in the two-scale sense}.
\end{align*}
Further, for every $\phi \in \mathcal{D}((0,T)\times  \overline{\Omega}, C_{\per}^{\infty}(\overline{Y^{\ast}}))$ and $\peps(t,x):= \phi\left(t,x,\fxe\right) $ it holds that (up to a subsequence)
\begin{align*}
\lim_{\epsilon \to 0} \int_0^T \langle \partial_t \veps , \peps \rangle_{H^1(\oe)',H^1(\oe)} dt = \int_0^T \langle \partial_t v_0 , \phi \rangle_{L^2(\Omega,H^1_{\per}(Y^{\ast})'),L^2(\Omega,H^1_{\per}(Y^{\ast}))} dt.
\end{align*}
\end{proposition}
\begin{proof}
By standard two-scale compactness results, a function $v_0 \in L^2((0,T)\times \Omega , H^1_{\per}(Y^{\ast}))$ exists such that (up to a subsequence) 
\begin{align*}
\veps &\rightarrow v_0 &\mbox{ in the two-scale sense,}
\\
\epsilon \nabla \veps &\rightarrow \nabla_y v_0 &\mbox{ in the two-scale sense}.
\end{align*}
To establish the existence of the weak time-derivative of $v_0$ we show that $\partial_t^h v_0 $ is bounded in $L^2((0,T-h),L^2(\Omega,H_{\per}^1(Y^{\ast})'))$ uniformly with respect to $h$.  
For  $\phi \in \mathcal{D}((0,T)\times  \Omega, C_{\per}^{\infty}(\overline{Y^{\ast}}))$ we define  $\peps (t,x):= \phi\left(t,x,\fxe\right)$. Then it holds that
\begin{align*}
\langle \partial_t^h v_0 &, \phi \rangle_{L^2((0,T-h),L^2(\Omega,H_{\per}^1(Y^{\ast})')),L^2((0,T-h),L^2(\Omega,H_{\per}^1(Y^{\ast})))}
\\
&= \int_0^{T-h} \int_{\Omega} \int_{Y^{\ast}} \partial_t^h v_0(t,x,y) \phi(t,x,y) dy dx dt 
\\
&= \lim_{\epsilon \to 0 } \int_0^{T-h} \int_{\oe} \partial_t^h \veps (t,x) \peps(t,x) dx dt 
\\
&= \lim_{\epsilon \to 0} \int_0^{T-h} \langle \partial_t^h \veps(t) , \peps(t,\cdot) \rangle_{\he',\he} dt
\\
&\le \lim_{\epsilon \to 0} \|\partial_t^h \veps \|_{L^2((0,T-h),\he')} \|\peps\|_{L^2((0,T-h),\he)}
\\
&\le C \lim_{\epsilon \to 0} \|\peps\|_{L^2((0,T-h),\he)}. 
\end{align*}
The smoothness of $\phi$ immediately gives 
\begin{align*}
\|\peps\|_{L^2((0,T-h),\he)}^2 &= \int_0^{T-h} \int_{\oe} \left\vert \phi\left(t,x,\fxe\right)\right\vert^2 + \epsilon^2 \left\vert \nabla_x \phi\left(t,x,\fxe\right)\right\vert^2  + \left\vert \nabla_y \phi \left(t,x,\fxe\right)\right\vert^2 dx dt
\\
&\overset{\epsilon\to 0}{\longrightarrow} \Vert \phi \Vert_{L^2((0,T -h),L^2(\Omega,H^1(Y^{\ast})))}^2, 
\end{align*}
yielding 
\begin{align*}
\left\vert \langle \partial_t^h v_0 , \phi \rangle_{L^2((0,T-h),L^2(\Omega,H_{\per}^1(Y^{\ast})')),L^2((0,T-h),L^2(\Omega,H_{\per}^1(Y^{\ast})))} \right\vert \le C \Vert \phi \Vert_{L^2((0,T),H^1(Y^{\ast}))}.
\end{align*}
By density arguments, this result holds for all $\phi \in L^2((0,T-h),L^2(\Omega,H_{\per}^1(Y^{\ast})))$. Since this space is reflexive, we obtain
\begin{align*}
\|\partial_t^h v_0 \|_{L^2((0,T-h),L^2(\Omega,H_{\per}^1(Y^{\ast})'))} \le C
\end{align*}
uniformly with respect to $h$. Hence, $\partial_t v_0 \in L^2((0,T),L^2(\Omega,H_{\per}^1(Y^{\ast})'))$. Further, for $\phi \in \mathcal{D}\big((0,T) \times \overline{\Omega}, C_{\per}^{\infty}(\overline{Y^{\ast}}))\big)$ and $\peps (t,x):=\phi\left(t,x,\fxe\right)$ we obtain by integration by parts that 
\begin{align*}
\int_0^T \langle \partial_t \veps , \peps \rangle_{H^1(\oe)',H^1(\oe)} dt  &= - \int_0^T \int_{\oe} \veps (t,x) \partial_t\phi\left(t,x,\fxe\right) dx dt 
\\
&\overset{\epsilon \to 0}{\longrightarrow} -\int_0^T \int_{\Omega}\int_{Y^{\ast}} v_0(t,x,y) \partial_t  \phi(t,x,y) dy dx dt 
\\
&= \int_0^T \langle \partial_t v_0 ,\phi \rangle_{L^2(\Omega,H^1_{\per}(Y^{\ast})'),L^2(\Omega,H^1_{\per}(Y^{\ast}))}.
\end{align*}
With this, the proposition is proved. 
\end{proof}

\begin{remark}
If the condition for the time-derivative $\partial_t \veps$ in Proposition \ref{TimeDerivativeTSConvergence} is replaced by the weaker one, namely $\partial_t \veps \in L^2((0,T),H^1(\oe)')$ with
\begin{align*}
\Vert \partial_t \veps \Vert_{L^2((0,T),H^1(\oe)') } \le C,
\end{align*}
the existence of $\partial_t v_0$ is not guaranteed anymore. However, using similar arguments as in the proof above, one shows that the time-derivative of $\bar{v}_0:= \int_{Y^{\ast}} v_0 dy $ exists. More precisely, we have $\bar{v}_0 \in L^2((0,T),H^1(\Omega)')$ and  for all $\phi \in H^1(\Omega)$ it holds that
\begin{align*}
\lim_{\epsilon\to 0} \int_0^T \langle \partial_t \veps , \phi\rangle_{H^1(\oe)',H^1(\oe)} = \int_0^T \langle \partial_t \bar{v}_0 , \phi \rangle_{H^1(\Omega)',H^1(\Omega)}.
\end{align*}
\end{remark}

For the proof of the strong two-scale compactness result we make use of the following lemma, which gives a relation between differences of shifted unfolded functions and the functions itself.

\begin{lemma}\label{EstimateShiftsUnfoldingAndSequence}
Let $\veps \in L^2((0,T)\times \oe)$. For $0 < h \ll 1$  and $|z|< h$ it holds that
\begin{align*}
\big\|\te \veps (t,x + z , y) - \te \veps \big\|_{L^2(0,T) \times \Omega_{2h} \times Y^{\ast})}^2 \le \sum_{j \in \{0,1\}^n} \|\delta_l \veps \|_{L^2((0,T)\times \oeh)}^2,
\end{align*}
with $l = l(\epsilon,z,j) = j + \left[\frac{z}{\epsilon}\right]$.
\end{lemma}
\begin{proof}
This result was proved for thin domains in \cite[p. 709-710]{NeussJaeger_EffectiveTransmission} and extended to domains in \cite[Proof of Theorem 3]{GahnNeussRaduKnabnerEffectiveModelSubstrateChanneling}.
\end{proof}

Now, we are able to formulate the strong two-scale compactness result.
\begin{theorem}
\label{StrongTwoScaleCompactness}
Consider a sequence of functions $\veps \in L^2((0,T),H^1(\oe))\cap H^1((0,T),\he')$ satisfying the following conditions 
\begin{enumerate}
[label = (\roman*)]
\item\label{StrongTSConvergenceConditionAprioriEstimate} An $\epsilon$-independent $C > 0$ exists such that 
\begin{align*}
\|\veps \|_{L^2((0,T),\he)} +  \|\partial_t \veps \|_{L^2((0,T),\he')} \le C. 
\end{align*}
\item\label{StrongTSConvergenceConditionShifts} For $0 < h \ll 1 $ and $\ell\in \Z^n$ with $|\ell\epsilon|< h$, it holds that
\begin{align*}
\|\delta \veps \|_{L^2((0,T),L^2(\oeh))} + \epsilon \|\nabla \delta \veps \|_{L^2((0,T),L^2(\oeh))} \overset{\epsilon l \to 0}{\longrightarrow} 0.
\end{align*}
\end{enumerate}
Then, there exists $v_0 \in L^2((0,T)\times \Omega , H_{\per}^1(Y^{\ast}))$, such that for $\beta \in \left(\frac12,1\right)$ and $p\in [1,2)$ it holds that
\begin{align*}
\te \veps \rightarrow v_0 \quad \mbox{ in } L^p(\Omega,L^2((0,T),H^{\beta}(Y^{\ast}))).
\end{align*}
\end{theorem}

\begin{remark}
The strong convergence of $\te \veps $ to $v_0$ in $L^2((0,T)\times \Omega , H^{\beta}(Y^{\ast}))$ can be obtained if the condition \ref{StrongTSConvergenceConditionShifts} is replaced by 
\begin{align*}
\|\delta \veps \|_{L^2((0,T),L^2(\oeh))} + \epsilon \|\nabla \delta \veps \|_{L^2((0,T),L^2(\oeh))} \overset{h \to 0}{\longrightarrow} 0,
\end{align*}
where $|\ell\epsilon|< h$. For this, in the proof below one has to make use of a result that is similar to \cite[Theorem 2.2]{GahnNeussRaduKolmogorovCompactness}, which can be proved in the same way as there. However, to obtain the results stated here, the strong convergence for $p \in [1, 2)$ is sufficient. 
\end{remark}

\begin{proof}[Proof of Theorem \ref{StrongTwoScaleCompactness}]
We apply \cite[Corollary 2.5]{GahnNeussRaduKolmogorovCompactness} , which gives a generalization of Simon compactness results from \cite{Simon} for domains in $\R^n$, and therefore we have to check the following properties:
\begin{enumerate}
[label = (\alph*)]
\item\label{ConditionEstimate} For $A \subset \Omega$ measurable, the sequence $V_{\epsilon}(t,y):= \int_A \te \veps (t,x,y)dx $ is relatively compact in $L^2((0,T),H^{\beta}(Y^{\ast}))$.
\item\label{ConditionShift} For $0<h\ll 1 $ and $|z|<h$ it holds that
\begin{align*}
\sup_{\epsilon} \big\|\te \veps (t,x + z,y) - \te \veps \big\|_{L^p(\Omega^h, L^2((0,T),H^{\beta}(Y^{\ast})))} \overset{z \to 0 }{\longrightarrow} 0.
\end{align*}
\item\label{ConditionBoundary} It holds that ($h>0$)
\begin{align*}
\sup_{\epsilon} \big\| \te \veps \big\|_{L^p(\Omega \setminus \Omega^h , L^2((0,T),H^{\beta}(Y^{\ast})))} \overset{h\to 0}{\longrightarrow} 0. 
\end{align*}
\end{enumerate}
First of all, let $A\subset \Omega$ and $V_{\epsilon} $ defined as in \ref{ConditionEstimate}. The properties of the unfolding operator and the condition  \ref{StrongTSConvergenceConditionAprioriEstimate} imply
\begin{align*}
\|V_{\epsilon}\|_{L^2((0,T),H^1(Y^{\ast}))}^2 \le \big\|\te \veps \big\|^2_{L^2((0,T)\times \Omega , H^1(Y^{\ast}))} \le C\|\veps\|_{L^2((0,T),\he)}^2\le C.
\end{align*}
Further, it holds for all $\phi \in \hoy$ that
\begin{align*}
\langle \partial_t V_{\epsilon} , \phi \rangle_{\hoy',\hoy} = \big\langle \partial_t \te \veps , \chi_A \phi \big\rangle_{L^2(\Omega,\hoy)',L^2(\Omega,\hoy)},
\end{align*}
where $\chi_A$ denotes the characteristic function on $A$. This implies for $\phi \in \hoy$ with $\|\phi\|_{H^1(Y^{\ast})} \le 1$ together with Proposition \ref{DerivativeUnfoldingOperator}
\begin{align*}
\big|\langle \partial_t V_{\epsilon} , \phi \rangle_{\hoy',\hoy}\big| &\le \big\|\partial_t \te \veps \big\|_{L^2(\Omega,\hoy)'} \|\chi_A \phi\|_{L^2(\Omega,\hoy)}
\le C\|\partial_t \veps \|_{\he}.
\end{align*}
Now, the condition \ref{StrongTSConvergenceConditionAprioriEstimate} implies that 
\begin{align*}
\|\partial_t V_{\epsilon}\|_{L^2((0,T),\hoy)'} \le C,
\end{align*}
\ie $V_{\epsilon}$  is bounded in $L^2((0,T),H^1(Y^{\ast}))\cap H^1((0,T),\hoy')$. Since the embedding $H^1(Y^{\ast})\hookrightarrow H^{\beta}(Y^{\ast}) $ is compact for $\beta \in \left(\frac12,1\right)$ whereas $H^{\beta}(Y^{\ast}) \hookrightarrow \hoy'$ continuously, the Aubin-Lions Lemma implies that $V_{\epsilon}$ is relatively compact in $L^2((0,T),H^{\beta}(Y^{\ast}))$. This is condition \ref{ConditionEstimate} above. 

To prove \ref{ConditionShift}, we use Lemma \ref{EstimateShiftsUnfoldingAndSequence}  to obtain that for $0<h \ll 1$ and $|z|<h$ and $\epsilon $ small enough
\begin{align*}
\big\|\te \veps &(t,x + z , y) - \te \veps \big\|_{L^p(\Omega_{2h},L^2((0,T),H^{\beta}(Y^{\ast})))}^2 
\\
&\le C \big\|\te \veps(t,x+z ,y ) -\te \veps \|^2_{L^2((0,T)\times \Omega_{2h},H^1(Y))}
\\
&\le C\sum_{j\in \{0,1\}^n} \|\delta \veps \|_{L^2((0,T)\times \oeh)}^2 + C \epsilon^2 \| \nabla \delta\veps\|_{L^2((0,T)\times \oeh)}^2 
\end{align*}
with $l= \epsilon \left(j + \left[\frac{z}{\epsilon}\right] \right)$  (in the definition of $\delta$). Hence, for $\epsilon,z \to 0$ we have $\ell\epsilon \to 0$. Due to the condition \ref{StrongTSConvergenceConditionShifts} the right-hand side converges to $0$ for $\epsilon,z\to 0$. This means, that \ref{ConditionShift} holds for all but finitely many $\epsilon$.  However, for these finitely many $\epsilon$, we can use the standard Kolmogorov compactness result, and therefore \ref{ConditionShift} is proved.

Condition \ref{ConditionBoundary} is an easy consequence of the H\"older inequality. For $p^{\ast} := \frac{2p}{2-p}$, 
\begin{align*}
\|\te \veps\|_{L^p(\Omega\setminus \Omega^h , L^2((0,T),H^{\beta}(Y^{\ast})))} \le \left|\Omega \setminus \Omega^h\right|^{p^{\ast}} \|\te \veps \|_{L^2((0,T) \times \Omega,H^1(Y))} \le Ch^{p^{\ast}} \overset{h\to 0}{\longrightarrow} 0,
\end{align*}
where in the last inequality we used the condition \ref{StrongTSConvergenceConditionAprioriEstimate}.
\end{proof}

\section{Derivation of the macroscopic model}
\label{DerivationMacroscopicModel}

In this section we use the compactnes results obtained in Section \ref{StrongTwoScaleResult} and the \textit{a priori} estimates from Section \ref{SectionAprioriEstimates} to derive the macroscopic model, obtained for $\epsilon \to 0$. In a first step, we derive an effective model defined on the reference domain, more precisely the reference element. Eventually, we transform the model to the one defined on a moving domain, described with the help of time- and space-dependent reference elements.

\begin{proposition}
\label{ConvergenceResults}
Let $\ueps$ be a sequence of weak solutions to Problem P$_T$ in $\eqref{MicProblemFixedDomain}$.
There exists a $u_0\in L^2((0,T)\times \Omega,H^1_{\per}(Y^{\ast})))\cap L^{\infty}((0,T),L^2(\Omega\times Y^{\ast}))$ with $J_0 u_0 \in L^2((0,T)\times \Omega,H^1_{\per}(Y^{\ast}))\cap H^1((0,T),L^2(\Omega,H^1_{\per}(Y^{\ast})'))$, such that, up to a subsequence, for $p\in [1,2)$ and $\beta \in \left(\frac12,1\right)$ it holds that
\begin{align*}
\te \ueps &\rightarrow u_0 &\mbox{ in }& L^p(\Omega,L^2((0,T)\times Y^{\ast})),
\\
\te \ueps &\rightarrow u_0 &\mbox{ in }& L^p(\Omega,L^2((0,T)\times \Gamma)),
\\
\te (\jeps \ueps) &\rightarrow J_0 u_0 &\mbox{ in }& L^p(\Omega,L^2((0,T),H^{\beta}(Y^{\ast}))).
\end{align*}
Additionally, for every $\phi \in \mathcal{D}\big((0,T),L^2(\Omega) , H_{\per}^1(Y^{\ast})\big)$ and $\peps(t,x):= \phi\left(t,x,\fxe\right) $
\begin{align*}
\lim_{\epsilon \to 0} \int_0^T \langle \partial_t (\jeps \ueps) , \peps \rangle_{H^1(\oe)',H^1(\oe)} dt = \int_0^T \langle \partial_t (J_0 u_0) , \phi \rangle_{L^2(\Omega,H^1_{\per}(Y^{\ast})'),L^2(\Omega,H^1_{\per}(Y^{\ast}))} dt.
\end{align*}
\end{proposition}
Before giving the proof of Proposition \ref{ConvergenceResults}, we state briefly a result about the strong convergence of $\jeps$ and the regularity of the time-derivative $\partial_t J_0$. This  follows directly from the assumptions on $\jeps$. 
\begin{remark}\label{StrongConvergenceJeps} 
Theorem \ref{StrongTwoScaleCompactness} implies that 
\begin{align*}
\te \jeps \rightarrow J_0 \quad \mbox{ in } L^2((0,T)\times \Omega,H^{\beta}(Y^{\ast}))
\end{align*}
for $\beta \in \left(\frac12,1\right)$.   Further, by Proposition \ref{TimeDerivativeTSConvergence} and the boundedness of $\te \jeps $ in $L^{\infty}((0,T)\times \Omega, W^{1,p}(Y^{\ast}))$, for every $p \in [1,\infty)$ we obtain
\begin{align*}
J_0 &\in L^{\infty}((0,T)\times \Omega, W^{1,p}_{\per}(Y^{\ast})),
\\
\nabla J_0 &\in L^{\infty}((0,T)\times \Omega \times Y^{\ast}),
\\ 
\partial_t J_0 &\in L^2((0,T),L^2(\Omega,H_{\per}^1(Y^{\ast})')),
\\
c_0 &\le J_0 \le C_0.
\end{align*} 
\end{remark}

\begin{proof}[Proof of Proposition \ref{ConvergenceResults}]
Assumption \ref{AssumptionConvergenceTransformation} gives $\te \jeps \rightarrow J_0$ in $L^q((0,T)\times \Omega \times Y^{\ast})$ for every $q \in [1,\infty)$. The a priori estimates from Lemma \ref{AprioriEstimates} imply the existence of a $u_0 \in L^2((0,T)\times \Omega, H^1_{\per}(Y^{\ast}))$, such that up to a subsequence
\begin{align*}
\te \ueps &\rightharpoonup u_0 \quad \mbox{ weakly in } L^2((0,T)\times \Omega,H^1(Y^{\ast})).
\end{align*}
Further, since $\te \ueps \in L^{\infty}((0,T),L^2(\Omega \times Y^{\ast})$ it holds that (see again Lemma \ref{AprioriEstimates})
\begin{align*}
\Vert u_0 \Vert_{L^p((0,T),L^2(\Omega \times Y^{\ast}))} \le \lim_{\epsilon \to 0 } \Vert \te \ueps \Vert_{L^p((0,T),L^2(\Omega \times Y^{\ast}))} \le C
\end{align*}
for all $p \in [1,\infty)$ with a constant $0 < C$ independent of $p$. Hence, we have obtained that $u_0 \in L^{\infty}((0,T),L^2(\Omega \times Y^{\ast}))$.

With  $\te \jeps \in L^{\infty}((0,T)\times \Omega \times Y^{\ast})$ and since it converges strongly in the $L^q$ sense for any $q\in [1,\infty)$, we obtain 
\begin{align*}
\te (\jeps \ueps ) \rightharpoonup J_0 u_0 \quad \mbox{ weakly in } L^2((0,T)\times \Omega \times Y^{\ast}).
\end{align*}

Further, due to Lemmata \ref{AprioriEstimates} and   \ref{ApriorEstimatesShifts}, recalling Assumption \ref{AssumptionInitialConditions}, the sequence $\jeps \ueps$ fulfills the conditions of Proposition \ref{TimeDerivativeTSConvergence} and Theorem \ref{StrongTwoScaleCompactness}. Hence $J_0 u_0 \in L^2((0,T)\times \Omega ,H^1_{\per}(Y^{\ast}))\cap H^1((0,T),L^2(\Omega,H^1_{\per}(Y^{\ast})'))$ and for $p \in [1,2)$ and $\beta \in \left(\frac12,1\right)$ it holds that, up to a subsequence, 
\begin{align*}
\te (\jeps \ueps) \rightarrow J_0 u_0 \quad \mbox{ in } L^p(\Omega, L^2((0,T),H^{\beta}(Y^{\ast}))).
\end{align*}
Moreover, for every $\phi \in \mathcal{D}\big((0,T) \times \overline{\Omega} , H_{\per}^1(Y^{\ast})\big)$ and $\peps(t,x):= \phi\left(t,x,\fxe\right) $, 
\begin{align*}
\lim_{\epsilon \to 0} \int_0^T \langle \partial_t (\jeps \ueps) , \peps \rangle_{H^1(\oe)',H^1(\oe)} dt = \int_0^T \langle \partial_t (J_0 u_0) , \phi \rangle_{L^2(\Omega,H^1_{\per}(Y^{\ast})'),L^2(\Omega,H^1_{\per}(Y^{\ast}))} dt.
\end{align*}
Especially, due to the continuity of the embedding $H^{\beta}(Y^{\ast}) \hookrightarrow L^2(\Gamma)$ we obtain
\begin{align*}
\te (\jeps\ueps) \rightarrow  J_0u_0 \quad \mbox{ in } L^p(\Omega,L^2((0,T) \times \Gamma)).
\end{align*}
Now, since $\jeps \geq c >0$, the strong convergence of $\te \jeps$ (which holds in $L^2((0,T)\times \Omega ,H^{\beta}(Y^{\ast})))$, see Remark \ref{StrongConvergenceJeps}) and Lebesgue's  Dominated Convergence Theorem imply 
\begin{align*}
\te \ueps &\rightarrow u_0 &\mbox{ in }& L^p(\Omega,L^2((0,T)\times Y^{\ast})),
\\
\te \ueps &\rightarrow u_0 &\mbox{ in }& L^p(\Omega,L^2((0,T)\times \Gamma)).
\end{align*}
\end{proof}

We remark that for the time-derivative $\partial_t u_0$ we obtain the following regularity: From the product rule we obtain in the distributional sense
\begin{align*}
\partial_t u_0 = J_0^{-1} \partial_t (J_0 u_0) - J_0^{-1} u_0 \partial_t J_0.
\end{align*}
Since $J_0^{-1} \in L^{\infty}((0,T)\times \Omega, W^{1,p}_{per}(Y^{\ast}))$ for every $p \in [1,\infty)$, and especially $p > n$, we have for every $\phi \in H_{\per}^1(Y^{\ast})$ that $J_0^{-1}(t,x,\cdot_y) \phi \in H_{\per}^1(Y^{\ast})$ for almost every $(t,x) \in (0,T)\times \Omega$. Hence, from $\partial_t (J_0 u_0 ) \in L^2((0,T)\times \Omega ,H^1_{\per}(Y^{\ast})')$ we obtain that the first term on the right in the equation above is an element of $L^2((0,T)\times \Omega,H^1_{\per}(Y^{\ast})')$. However, this is not true for the second term. We only have that, for $\phi \in W_{\per}^{1,q}(Y^{\ast})$ with $q > n$, 
\begin{align*}
\langle J_0^{-1} u_0 \partial_t J_0, \phi &\rangle_{W_{\per}^{1,q}(Y^{\ast})',W_{\per}^{1,q}(Y^{\ast})} = \langle \partial_t J_0, J_0^{-1} u_0 \phi \rangle_{H_{\per}^1(Y^{\ast})',H_{\per}^1(Y^{\ast})},
\end{align*}
almost everywhere in $(0,T)\times \Omega $. Therefore $\partial_t u_0 \in L^1((0,T)\times \Omega,W^{1,q}_{\per}(Y^{\ast})')$.

\begin{corollary}\label{ConvergenceResultsNonlinearities}
Up to a subsequence, it holds that
\begin{align*}
f(\ueps) &\rightarrow f(u_0) &\mbox{ in the two-scale sense,}
\\
g(\ueps) &\rightarrow g(u_0) &\mbox{ in the two-scale sense on } \ge.
\end{align*}
\end{corollary}
\begin{proof}
This is an easy consequence of the strong convergence results for $\ueps$ from Proposition \ref{ConvergenceResults}. For more details see \cite[Corollary 5]{GahnNeussRaduKnabner2018a}.
\end{proof}

The assumptions on the transformation $\seps$ guarantee that  
\begin{align*}
\deps &\rightarrow \left[\nabla_y S_0\right]^{-1} D \left[\nabla_y S_0\right]^{-T} &\mbox{ strongly in the two-scale sense},
\\
\epsilon^{-1}\veps &\rightarrow \left[\nabla_y S_0\right]^{-1} \partial_t S_0 &\mbox{ strongly in the two-scale sense,}
\end{align*}
where the strong two-scale convergence is valid with respect to every $L^p$-norm for $p\in [1,\infty)$. To simplify the writing we define 
\begin{align*}
D_0^{\ast}:= \left[\nabla_y S_0\right]^{-1} D \left[\nabla_y S_0\right]^{-T}, \quad q_0^{\ast}:= \left[\nabla_y S_0 \right]^{-1} q_0 ,
 \quad  v_0^{\ast}:= \left[\nabla_y S_0\right]^{-1} \partial_t S_0 
\end{align*}
and state \textbf{Problem P$_M$}, which is to find $u_0$ solving 
\begin{align*}
\partial_t (J_0 u_0) - \nabla_y \cdot \left( J_0 D_0^{\ast} \nabla_y u_0 - J_0 q_0^{\ast} u_0 + J_0 v_0^{\ast} u_0 \right) &=J_0 f(u_0) &\mbox{ in }& (0,T)\times \Omega \times Y^{\ast},
\\
-  J_0 D_0^{\ast} \nabla_y u_0  \cdot \nu &= -J_0 g(u_0)  &\mbox{ on }& (0,T)\times \Omega \times \Gamma, 
\\
u_0(0) &= u^0 &\mbox{ in }& \Omega \times Y^{\ast},
\\
u_0(t,x,\cdot) \mbox{ is } Y&\mbox{-periodic}.
\end{align*}
Due to the low regularity of $\partial_t u_0 $, we cannot guarantee $u_0 \in C^0([0,T], L^2(\Omega \times Y^{\ast}))$, so the initial condition $u_0 (0) = u^0$ holds only in a weaker sense. In fact, we show that there is a set of measure zero $N \subset (0,T)$, such that  
\begin{align*}
\lim_{t \to 0 ,\, t\notin N} \Vert u_0(t) - u^0 \Vert_{L^1(\Omega \times Y^{\ast})} = 0.
\end{align*}
Problem P$_M$ is the macroscopic counterpart of Problem P$_T$, as follows from  
\begin{theorem}\label{TheoremMacroscopicProblem}
The limit function $u_0 \in  L^2((0,T)\times \Omega,H^1_{\per}(Y^{\ast}))$ satisfies $\partial_t (J_0 u_0) \in L^2((0,T),L^2(\Omega,H_{\per}^1(Y^{\ast})'))$, and is the unique weak solution of 
Problem P$_M$. 
\end{theorem}

\begin{proof}

Testing in $\eqref{VarEquaMicProblemFixedDomain} $ with $\peps(t,x):= \phi\left(t,x,\fxe\right)$ for $\phi \in C_0^{\infty}\big([0,T)\times \Omega,C_{\per}^{\infty}(Y^{\ast})\big)$ gives 
\begin{align*}
&\int_0^T \langle \partial_t (\jeps \ueps) , \peps \rangle_{\oe} dt 
\\
+& \int_0^T \int_{\oe} \left[\epsilon^2 \jeps \deps \nabla \ueps - \epsilon \jeps \qeps \ueps + \jeps \veps \ueps \right] \cdot \left[\nabla_x \phi\left(t,x,\fxe\right) + \frac{1}{\epsilon}\nabla_y\phi\left(t,x,\fxe\right) \right] dx dt 
\\
=& \int_0^T \int_{\oe} \jeps f(\ueps) \peps dx dt + \epsilon \int_0^T \int_{\ge} \jeps g(\ueps) \peps d\sigma dt.
\end{align*}
Using the convergence results from Proposition \ref{ConvergenceResults}, we obtain for $\epsilon \to 0$
\begin{align*}
\int_0^T &\langle \partial_t (J_0 u_0 ) , \phi \rangle_{L^2(\Omega,H_{\per}^1(Y^{\ast})'),L^2(\Omega,H_{\per}^1(Y^{\ast}))}dt
\\
&+ \int_0^T \int_{\Omega} \int_{Y^{\ast}} \left[ J_0 D_0^{\ast} \nabla_y u_0 - J_0 q_0^{\ast} u_0 + J_0  v_0^{\ast} u_0 \right] \cdot \nabla_y \phi dy dx dt
\\
=& \int_0^T\int_{\Omega}\int_{Y^{\ast}} J_0 f(u_0) \phi dy dx dt + \int_0^T \int_{\Omega}\int_{Y^{\ast}} J_0 g(u_0) \phi d\sigma_y dx dt.
\end{align*}
The regularity of $J_0 u_0$ stated in Proposition \ref{ConvergenceResults} implies that $J_0 u_0 \in C^0([0,T],L^2(\Omega \times Y^{\ast}))$. Hence, integration by parts with respect to time in the equations above, and similar arguments as before, imply 
\begin{align*}
(J_0 u_0) (0) = J_0(0) u^0. 
\end{align*}
Now, for almost every $t \in (0,T)$ we have  
\begin{align*}
\Vert u_0(t) - u^0 \Vert_{L^1(\Omega \times Y^{\ast})} \le C \big(\Vert J_0(0) u_0(t) - J_0(t) u_0(t)\Vert_{L^1(\Omega \times Y^{\ast})} + \Vert J_0(t)u_0(t) -  J_0(0) u^0 \Vert_{L^1(\Omega \times Y^{\ast})} \big).
\end{align*}
The second term vanishes for $t \to 0$, due to the observations above. For the first term we obtain  using $u_0 \in L^{\infty}((0,T),L^2(\Omega \times Y^{\ast}))$
\begin{align*}
\Vert J_0(0)u_0(t) - J_0(t) u_0(t)\Vert_{L^1(\Omega \times Y^{\ast})} &\le \Vert J_0(0) - J_0(t)\Vert_{L^2(\Omega \times Y^{\ast})} \Vert u_0(t)\Vert_{L^2(\Omega \times Y^{\ast})}
\\
&\le C \Vert J_0(0) - J_0(t)\Vert_{L^2(\Omega \times Y^{\ast})} \overset{t \to 0}{\longrightarrow } 0 .
\end{align*}
\end{proof}

\subsection{Transformation to the evolving macroscopic domain}
Now we formulate the macroscopic Problem P$_M$ 
in an evolving macroscopic domain. We define $Y(t,x):= S_0(t,x,Y^{\ast})$ and $\Gamma(t,x):= \partial Y(t,x) \setminus \partial Y$, and
\begin{align*}
\tuo: \bigcup_{(t,x)\in (0,T)\times \Omega} \{(t,x)\} \times Y(t,x) \rightarrow \R, \quad \tuo (t,x,y) = u_0\left(t,x,S_0^{-1}(t,x,y)\right). 
\end{align*}
Further, we set for $(t,x)\in (0,T)\times \Omega$ and $y \in Y(t,x)$
\begin{align*}
 \tilde{q}_0(t,x,y):= q_0\left(t,x,S_0^{-1}(t,x,y)\right),
\end{align*}
and for $x\in \Omega$ and $y \in Y(0,x)$ the initial condition
\begin{align*}
\tilde{u}^0(x,y):= u^0\left(x,S_0^{-1}(0,x,y)\right).
\end{align*}
Finally, let $Q^T:= \bigcup_{(t,x)\in (0,T)\times \Omega} \{(t,x)\} \times Y(t,x)$ and $G^T:= \bigcup_{(t,x)\in (0,T)\times \Omega} \{(t,x)\} \times \Gamma(t,x)$. With this, an elemental calculation shows that for all $\phi \in C^1\left(\overline{Q^T}\right)$ with $\phi(T,\cdot) = 0$ and $Y$-periodic, one has  
\begin{align*}
-&\int_0^T \int_{\Omega}\int_{Y(t,x)} \tuo \partial_t \phi dy dx dt + \int_0^T \int_{\Omega}\int_{Y(t,x)}\left[ D \nabla_y \tuo - \tilde{q}_0 \tuo \right] \cdot \nabla_y \phi dy dx dt 
\\
=& \int_0^T \int_{\Omega}\int_{Y(t,x)} f(\tuo)\phi dy dx dt + \int_0^T \int_{\Omega}\int_{\Gamma(t,x)} g(\tuo)\phi d\sigma_y dx dt
\\
&+ \int_{\Omega}\int_{Y(0,x)} \tilde{u}^0 \phi(0) dy dx .
\end{align*}
In other words, $\tuo$ is the weak solution of the macroscopic problem defined in an evolving macroscopic domain   
\begin{align}
\begin{aligned}\label{HomogenizedModelEvolving}
\partial_t \tuo - \nabla_y \cdot \left(D \nabla_y \tuo - \tilde{q}_0 \tuo \right) &= f(\tilde{u_0}) &\mbox{ in }& Q^T,
\\
- D \nabla_y \tuo \cdot \nu &= - g(\tuo) &\mbox{ on }& G^T,
\\
\tuo(0) &= \tilde{u}^0 &\mbox{ in }& \bigcup_{x \in \Omega} \{x\} \times Y(0,x),
\\
\tuo \mbox{ is } Y&\mbox{-periodic.}
\end{aligned}
\end{align}
We emphasize that for regular enough data and solutions (which are not guaranteed by our assumptions), we would obtain $\tilde{q}_0(t,x,y) \cdot \nu = \partial_t S_0(t,x,y) \cdot \nu$ for almost every $(t,x,y) \in G^T$.

\section{Conclusion}
\label{SectionConclusion}

We derived a macroscopic model for a reaction-diffusion-advection problem defined in a domain with an evolving microstructure. The evolution is assumed known \textit{a priori}. We consider a low  diffusivity (of order $\epsilon^2$), and include nonlinear bulk and surface reactions. The effective problem depends on the micro- and the macro-variable, and the evolution of the underlying microstructure is approximated by time- and space-dependent reference elements $Y(t,x)$. Hence, in each macroscopic point $x$, we have to solve a local cell problem on $Y(t,x)$. We emphasize that our methods are not restricted to the scalar case, but can be extended easily to systems of equations with Lipschitz-continuous nonlinearities. In order to carry out the homogenization limit, we proved general two-scale compactness results  just based on \textit{a priori} estimates for sequences of functions with oscillating gradients, and low regularity with respect to time.  In doing so, we used the appropriately scaled function space $\he$, which allowed us to show compactness results, especially regarding the time-derivative. 

In general applications, however, the evolution of the microstructure is not known \textit{a priori}, and can be influenced by adsorption and desorption processes at the microscopic surface, or by mechanical forces. In such cases, one has to consider strongly coupled systems of transport, elasticity, and fluid flow equations in domains with an evolving microstructure, leading to highly nonlinear problems with free boundaries. In this context, the identification of the transformation $\seps$ and its control with respect to the parameter $\epsilon$ plays a crucial role. If this is achieved, the multi-scale methods developed in this paper can be employed in the study of such more complex applications. 

\section*{Acknowledgments}
MG and ISP were supported by the Research Foundation - Flanders (FWO) through the Odysseus programme (Project G0G1316N). MG was also supperted by the project SCIDATOS (Scientific Computing for Improved Detection and Therapy of Sepsis), which was funded by the Klaus Tschira Foundation, Germany (Grant Number 00.0277.2015).

  \bibliographystyle{abbrv} 
  \bibliography{literature}

\end{document}